\documentclass[final,reqno,twoside,10pt]{siamltex}

\usepackage{amssymb,amsfonts}
\usepackage{amsmath,xcolor}

\newcommand{\varep}{\varepsilon}

\newcommand{\norm} [1]{\left\| {#1}\right\|}
\newcommand{\zz} {\mathbf {z}}
\newcommand{\vv} {\mathbf {v}}
\newcommand{\uu} {\mathbf {u}}
\newcommand{\s} {\mathbf {s}}

\newcommand{\ddt}{\frac d{dt} }

\def\beq{\begin{equation}}
\def\eeq{\end{equation}} 
\def\beqs{\begin{equation*}}
\def\eeqs{\end{equation*}}

\def\myclearpage{}

\title{The expanded mixed finite element method for generalized Forchheimer flows in porous media\thanks{This 
        work was supported by NSF grant (DMS, Proposal ID: 1412796)}}

\author{Akif Ibragimov \thanks{Department of Mathematics and Statistics, Texas Tech University, Box 41042
Lubbock, TX 79409-1042, U.S.A. ({\tt akif.ibraguimov@ttu.edu}).}
             \and  Thinh T. Kieu\thanks{Department of Mathematics, University of North Georgia, Gainesville Campus, 3820 Mundy Mill Rd., Oakwood, GA 30566, U.S.A. (\tt  thinh.kieu@ung.edu)}}

\date{\today}

\begin{document}

\maketitle

\begin{abstract}
 We study the expanded mixed finite element method applied to degenerate parabolic equations with the Dirichlet boundary condition. The equation is considered a prototype of the nonlinear Forchheimer equation, a inverted to the nonlinear Darcy equation with permeability coefficient depending on pressure gradient, for slightly compressible fluid flow in porous media. The bounds for the solutions are established. In both continuous and discrete time procedures, utilizing the monotonicity properties of Forchheimer equation and boundedness of solutions we prove the optimal error estimates  in $L^2$-norm for solution. The error bounds are established for the solution and divergence of the vector variable in Lebesgue norms and Sobolev norms  under some additional regularity assumptions. A numerical example using the lowest order Raviart-Thomas ($RT_0$) mixed element are provided agreement with our theoretical analysis.
\end{abstract}

\begin{keywords} 
Expanded mixed finite element, nonlinear degenerate parabolic equations, generalized Forchheimer equations, error estimates.
\end{keywords}

\begin{AMS}
65M12, 65M15, 65M60, 35Q35, 76S05.
\end{AMS}

\pagestyle{myheadings}
\thispagestyle{plain}
\markboth{Akif Ibragimov and Thinh T. Kieu}{Expanded mixed finite element method for generalized Forchheimer equations}

\section{Introduction}
%-----------------------------------------------------------------------------------------------
%Section 3
%-----------------------------------------------------------------------------------------------
The paper is dedicated to the analysis of mixed finite element approximation of the solution of the system modeling  the flows of compressible fluid in porous media subjected to generalized Forchheimer law. Forchheimer type  flow belongs to the so called post-Darcy class of flows and are designed to model high velocity filtration in porous media when inertial and friction terms cannot be ignored. In recent years, this phenomena generated a lot of interest in the research community in many areas of engineering, environmental and groundwater hydrology and in medicine.

Accurate description of fluid flow behavior in porous media is essential to successful forecasting and project design in reservoir engineering.  Most of the analyses of the flow in porous media are based on Darcy law, which describes a linear relationship between the pressure gradient and the fluid velocity. However, when the fluid in porous media flows at very high or very low velocity, the Darcy law is no longer valid. Reservoir engineers  often divide flows in the media into three main categories with respect to Darcy law (cf. \cite{Darcy1}): fast flows near the well and fracture (post-Darcy), linear non-fast/non-slow flows described by Darcy equation in the main domain between near well zone and  ``far a way'' region, and on the periphery of  the media, where the impact of the well is small. A nonlinear relationship between velocity and gradient of pressure is introduced adding the higher order term of velocity in the Darcy equation.  In this research we concentrate on the first type of flow, when deviation from linear Darcy is associated with high velocity field.
 
Engineers commonly use Forchheimer equation to take into account inertial phenomena.  In early 1900s, Forchheimer proposed three models for nonlinear flows, the so-called two-term, three-term and power laws (cf. \cite{ForchheimerBook}) to match experimental observations. There is a significant number of papers studying these equations and their variations – the Brinkman-Forchheimer equations for incompressible fluids (cf.\cite{CKU2006,ChadamQin,Franchi2003,Straughan2013,Payne1999b,Qin1998}).  Recently the authors in \cite{ABHI1, HI1,HI2,HIKS1,HK,HKP1} proposed and studied generalized Forchheimer equations for slightly compressible fluids in porous media. These works focus on theory of existence, stability and qualitative property of solutions within framework of non-linear parabolic equation with coefficient degenerating as gradient of pressure converges to infinity. To apply developed models and method to practical problem, it is important to investigate properties and convergence of the approximate numerical solutions of corresponding degenerate parabolic equations.  

The popular numerical methods for modeling flow in porous media are mixed finite element approximations (cf. \cite {DW93, GW08, KP99, HRH12, EJP05}) and block-centered finite difference method (cf. \cite{HPH12}). These methods are widely used because they inherit conservation properties and because they produce accurate flux even for highly homogeneous media with large jumps of conductivity (permeability) tensor   (cf. \cite{ELPV96}). Arbogast, Wheeler and Zhang in \cite {ATWZ96} analyzed mixed finite element approximations of degenerate parabolic equation. However according to Arbogast, Wheeler and Yotov in \cite {ATWY97}, the standard mixed finite element method is not suitable for problems with degenerating tensor coefficients as the tensor needs to be inverted.

The proposed approach reduces original Forchheimer type equation to generalized Darcy equation with conductivity tensor $K$ degenerating as gradient of the pressure convergence to infinity. At the same time, the standard mixed variational formulation requires inverting $K$.  Woodward and Dawson in \cite{ WCD00} studied expanded mixed finite element methods for a nonlinear parabolic equation modeling flow into variably saturated porous media. Compared with the standard mixed finite element method, the expanded mixed finite element method introduces three variables: unknown scalar function, its gradient, and a flux. In their analysis, the Kirchhoff transformations are used to move the nonlinearity from $K$ term to the gradient and thus analysis of the equations is simplified. This transformation is  not applicable  to our system \eqref{syseq}. 

In this paper,  we will employ techniques developed in \cite {HI1} and the expanded mixed finite element method presented in \cite {ATWY97}. The combination of these techniques enables us to utilize both the special structure of the equation and the advantages of the expanded mixed finite element method, for instance, providing certain implementational advantages over the standard mixed method, in particularly for lowest order Raviart-Thomas (RT) space on the rectangular grids, in obtaining the optimal order error estimates for the solution in several norms of interest.
 
 The outline of this paper is as follows:  
 In section \ref{Pre-aux}, we introduce the generalized formulation of the Forchheimer’s laws for slightly compressible fluids, recall relevant results in \cite{ABHI1, HI1} and  preliminary results.  
 %-----------------------
 In section \ref{mixedFEM}, we present the expanded mixed formulation and standard results for mixed finite element approximations. An implicit backward-difference time discretization of the semidiscrete  scheme is proposed to solve the system \eqref{maineq}.
 %----------------------
 In section \ref{prior-est}, we derive many bounds for solutions to \eqref{weakform} and \eqref{semidiscreteform} in Lebesgue norms in term of boundary data and the initial data.
 %---------------------
 In section \ref{Err-anl}, we establish error estimates in $L^2$-norms, $L^\infty$-norm and $H^{-1}$-norm for pressure. Then the error estimates for gradient of pressure and flux variable are also derived under reasonable assumptions on the regularity of solutions. Also the error analysis for fully discrete version is obtained in suitable norm for the three relevant variables.  
 %----------------------------------
 In section \ref{num-res}, we give a numerical example using the lowest Raviart-Thomas mixed finite element to support our theoretical analysis. 

%========================

%========================
\section {Preliminaries and auxiliaries} \label{Pre-aux}
In this paper, we consider a fluid in a porous medium in a bounded domain $\Omega\subset \mathbb R^d, d\ge 2$.
 Its boundary $\Gamma=\partial \Omega$ belongs to $C^2$.  Let $x\in \mathbb{R}^d,    0<T<\infty,  t\in (0,T]$ be the spatial and time variable. The fluid flow has velocity $\uu(x,t)\in \mathbb{R}^d,$ pressure $p(x,t)\in \mathbb{R}$. 
 
 A general Forchheimer equation which is studied in \cite{ABHI1} has the form 
 \beq\label{eq1}
g(|\uu|)\uu=-\nabla p,
\eeq
where
$
g(s)=a_0+ a_1s^{\alpha_1}+\cdots +a_Ns^{\alpha_N},~~s\geq 0, 
$
$N\geq 1.$ $0<\alpha_1<\ldots<\alpha_N$ are fixed numbers, the coefficients $a_0, \ldots, a_N$ are non-negative numbers with $a_0>0, a_N>0$. In particular when $g(s)=\alpha, \alpha+\beta s, \alpha+\beta s+\gamma s^2 +\gamma_m s^{m-1}, $ where $\alpha, \beta, \gamma, m,\gamma_m$ are empirical constants, we have Darcy's law, Forheimer's two term, three term and power law, respectively. 

The monotonicity of the nonlinear term and the non-degeneracy of the Darcy's parts in \eqref{eq1} enable  us to write $\uu$ implicit in terms of $\nabla p$:
\beq\label{eq3}
\uu=-K(|\nabla p|)\nabla p. 
\eeq 
Here the function $K: \mathbb{R}^+\rightarrow\mathbb{R}^+$ is defined by
\beq\label{eq4}
K(\xi)=\frac{1}{g(s(\xi))} \text{~where~} s=s(\xi)\geq 0 \text{~satisfies~} sg(s)=\xi, \text{~~for~}\xi\geq 0.
\eeq
The state equation, which relates the density $\rho(x,t)>0$ with pressure $p(x,t)$, for slightly compressible fluids is 
   \beq\label{eq6}
   \frac{d\rho}{dp}=\kappa^{-1}\rho \text{~or~} \rho(p)=\rho_0\exp(\frac{p-p_0}{\kappa}),\quad \kappa>0.
   \eeq
Other equations governing the fluid's motion are the equation of continuity:
\beqs
\frac{d \rho}{d t} +\nabla\cdot(\rho \uu)=0,
\eeqs 
which yields 
\beq\label{eq5}
\frac {d\rho}{dp} \frac{d p}{d t} +\rho\nabla\cdot \uu + \frac {d\rho}{dp}\uu\cdot \nabla p=0. 
\eeq
Combining \eqref{eq6} and \eqref{eq5}, we obtain 
   \beq\label{eq7}
   \frac{d p}{d t}+\kappa \nabla\cdot \uu+\uu\cdot \nabla p=0.
   \eeq  
  Since the constant $\kappa$ is very large  for most slightly compressible fluids in porous media \cite{Muskatbook},  in most of the practical application the third term on the left-hand side of \eqref{eq7} is neglected. This results in the following reduced equation
\beq\label{Eq4pu} \frac{d p}{d t} +\kappa \nabla\cdot \uu=0.\eeq
By rescaling the time variable,  hereafter we assume that $\kappa=1$.  
From \eqref{eq3} and \eqref{Eq4pu} we have the system  
\beq\label{lin-p} 
\begin{cases}
\uu+K(|\nabla p|\nabla p) = 0,\\
p_t +\kappa \nabla\cdot \uu=0.
\end{cases}
\eeq
System \eqref{lin-p} gives a scalar equation for pressure: 
\beq p_t - \nabla\cdot ( K(|\nabla p|)\nabla p )=0.\eeq
The function $K(\xi)$ has the important properties (cf. \cite{ABHI1, HI1}): 

(i) $K: [0,\infty)\to (0,a_0^{-1}]$ and it decreases in $\xi,$  

(ii) Type of degeneracy  \beq\label{i:ineq1}  \frac{c_1}{(1+\xi)^a}\leq K(\xi)\leq \frac{c_2}{(1+\xi)^a},  \eeq

(iii) For all $n\ge 1,$ \beq\label{i:ineq2} c_3(\xi^{n-a}-1)\leq K(\xi)\xi^n\leq c_2\xi^{n-a}, \eeq

 (iv) Relation with its derivative \beq\label{i:ineq3} -aK(\xi)\leq K'(\xi)\xi\leq 0, \eeq
  where $c_1, c_2, c_3$ are positive constants depending on $\Omega$ and $g$. The constant $a\in(0,1)$ is defined by 
  \beq\label{a-const }
   a=\frac{\alpha_N}{\alpha_N+1}= \frac{\deg (g)}{\deg (g)+1}.
  \eeq
The following function is crucial in our estimates. We define 
\beq\label{Hdef}
H(\xi)=\int_0^{\xi^2} K(\sqrt{s}) dx, \text{~for~} \xi\geq 0. 
\eeq
The function $H(\xi)$ can compare with $\xi$ and $K(\xi)$ by
\beq
K(\xi)\xi^2 \leq H(\xi)\leq 2K(\xi)\xi^2, \label{i:ineq4}
\eeq
as a consequence of \eqref{i:ineq2} and \eqref{i:ineq4} we have,
\beq
C(\xi^{2-a}-1) \leq H(\xi) \leq 2C\xi^{2-a}. \label{i:ineq5}
\eeq 
%==================================================================%
For the monotonicity and continuity of the differential operator in \eqref{lin-p} we have the following results. 
\begin{proposition}[cf. \cite{HI1}] One has 

{\rm (i)} For all $y, y' \in \mathbb{R}^d$, 
\beq\label{Qineq}
\left(K(|y'|)y' -K(|y|)y \right)\cdot(y'-y)\geq (1-a)K( \max\{ |y|, |y'|\} )|y' -y|^2 .
\eeq      

 {\rm (ii)} For the vector functions $\s_1, \s_2$, there is a positive constant $C$ such that
\beq\label{Mono}
\int_\Omega \left(K(|\s_1|)\s_1-K(|\s_2|)\s_2\right)\cdot\left(\s_1 -\s_2\right) dx\geq C\omega\norm{\s_1-\s_2}_{L^\beta(\Omega)}^2,
\eeq
where 
\beq
\omega =\left(1+ \max\{\|\s_1\|_{L^\beta(\Omega)} ,  \|\s_2\|_{L^\beta(\Omega)} \}\right)^{-a}.
\eeq\end{proposition}
\begin{proposition}\label{Lips}  For all $y, y' \in \mathbb{R}^d$ we have 
\beq\label{Lipchitz}
   \left|K(|y'|)y' -K(|y|)y\right| \leq \sqrt {2(a^2+ 1)}a_0^{-1}|y' -y|.
\eeq  
\end{proposition}
\begin{proof}
{\it Case 1}: The origin does not belong to the segment connect  $y'$ and  $y$.  Let $\ell(t)=ty'+(1-t)y, t\in[0,1]$. Define $ h(t)=K(|\ell(t)|)\ell(t)$  for $t\in[0,1]$.
By the mean value theorem, there is $t_0\in[0,1]$ with $\ell(t_0)\neq 0$, such that
\begin{align*}
|K(|y'|)y' -K(|y|)y|^2 &= | h(1) - h(0)|^2  =|h'(t_0)|^2 \\
   &= \left| K'(|\ell(t_0)|)\frac{\ell(t_0)\cdot \ell'(t_0) }{|\ell(t_0)|}\ell(t_0)   +K(|\ell(t_0)|)\ell'(t_0) ) \right|^2.
\end{align*} 
Using \eqref{i:ineq3} and Minkowski's inequality we obtain    
\beqs   
|K(|y'|)y' -K(|y|)y|^2 \le 2 | K(|\ell(t_0)|)|^2\left(a^2 \left|\frac{\ell(t_0)\cdot \ell'(t_0) }{|\ell(t_0)|^2}\ell(t_0)  \right|^2   + |\ell'(t_0) )|^2\right).  
\eeqs
The \eqref{Lipchitz} follows by the boundedness of $K(\cdot)\le a_0^{-1}$.  

{\it Case 2}: The origin belongs to the segment connect $y', y$. We replace $y'$ by some $y_\varep\neq 0$ so that $0\notin [y_\varep, y]$ and $y_\varep \to 0$ as $\varep \to 0$. Apply the above inequality for $y$ and $y_\varep$, then let $\varep\to 0$.
\end{proof} 
%==================================================================%

{\bf Notations.} Let $L^2(\Omega)$ be the set of square integrable functions on $\Omega$ and $( L^2(\Omega))^d$ the space of $d$-dimensional vectors which have all components in $L^2(\Omega)$. 

We denote $(\cdot, \cdot)$ the inner product in either $L^2(\Omega)$ or $(L^2(\Omega))^d$ that is
\beqs
( \xi,\eta )=\int_\Omega \xi\eta dx \quad \text{ or } (\boldsymbol{\xi},\boldsymbol \eta )=\int_\Omega \boldsymbol{\xi}\cdot \boldsymbol{\eta} dx. 
\eeqs      
 The notation $\norm {\cdot}$ will means scalar norm $\norm{\cdot}_{L^2(\Omega)}$ or vector norm $\norm{\cdot}_{(L^2(\Omega))^d}$. 

For $1\le q\le +\infty$ and $m$ any nonnegative integer, let
\beqs
W^{m,q}(\Omega) = \{f\in L^q(\Omega), D^\alpha f\in L^q(\Omega), |\alpha|\le m \}
\eeqs 
denote a Sobolev space endowed with the norm 
\beqs
\norm{f}_{m,q} =
\left( \sum_{|\alpha|\le m} \norm{D^\alpha f}^q_{L^q(\Omega)} \right)^{\frac 1q}.
\eeqs    
Define $H^m(\Omega)= W^{m,2}(\Omega)$ with the norm $\norm{\cdot}_m =\norm{\cdot }_{m,2}$. 

For functions $p,u$ and vector-functions ${\mathbf v, } \s, \uu $ we use short hand notations  
\beqs
\norm{p(t)} = \norm{ p(\cdot, t)}_{L^2(\Omega)},\quad \norm{\uu(t)} = \norm{ \uu(\cdot, t)}_{L^2(\Omega)}, \quad \norm{\s(t)}_{L^\beta(\Omega)} = \norm{ \s(\cdot, t)}_{L^\beta(\Omega)} 
\eeqs  
and
\beqs
 u^0(\cdot) =  u(\cdot,0), \quad  {\mathbf v}^0(\cdot) =  {\mathbf v} (\cdot,0).  
\eeqs
Throughout this paper the constants 
$$\beta =2-a,\quad \gamma = \frac a{2-a}.$$ 
The arguments $C, C_1$ will represent for positive generic constants and their values  depend on exponents, coefficients of polynomial  $g$,  the spatial dimension $d$ and domain $\Omega$, independent of the initial and boundary data, size of mesh and time step. These constants may be different place by place. 

\section{The expanded mixed finite element methods}\label{mixedFEM}
%In this section, we derive the semidiscrete expanded mixed finite element method for the problem \eqref{syseq} and then a fully discrete version.  

We introduce the new variable $\s=\nabla p$ to \eqref{lin-p} and study the initial value boundary problem (IVBP):  
\beq\label{syseq}
\begin{cases}
     p_t +  \nabla \cdot \uu =f,\\
     \uu + K(|\s|)\s=0, \\
     \s- \nabla p = 0,
\end{cases}
\eeq 
 for all  $x\in  \Omega, t\in(0,T)$, where $f: \Omega\times(0,T)\to\mathbb R $, $f\in C^1([0,T];L^\infty(\Omega)).$ 

  The initial and boundary conditions:
\beqs
p(x,0)=p_0(x) \text {~in~} \Omega , \quad 
     p(x,t)=\psi(x,t) \text{ ~on~} \Gamma \times(0,T),
\eeqs
 we also require  at $t=0$:  $p_0(x) =\psi(x,0)$ on boundary $\Gamma$. 

To deal with the non-homogeneous boundary condition, we extend the Dirichlet boundary data  from boundary $\Gamma$ to the whole domain $\Omega$ (see \cite{HI1,JK95, MC70}). Let $\Psi(x,t)$ be a such extension.  
Let $\overline{p} = p-\Psi$. Then $\bar p(x,t)=0 \text{ ~on~} \Gamma \times(0,T).$ System \eqref{syseq} rewrites as  
\beq\label{maineq}
\begin{cases}
     \bar p_t +  \nabla \cdot \uu =-\Psi_t +f,\\
     \uu + K(|\s|)\s=0, \\
     \s- \nabla \bar p = \nabla \Psi, 
\end{cases}
\eeq 
for all  $(x,t)\in  \Omega\times (0,T)$, where $\bar{p}(x,0)=p_0(x)-\Psi(x,0)=\bar{p}_0(x).$ 
 
Define  $W=L^2(\Omega)$, $\tilde W= (L^2(\Omega))^d,$ and the Hilbert space    
$$V= H({\rm div}, \Omega)= \left\{ \vv \in  (L^2(\Omega))^d, \nabla \cdot \vv \in L^2(\Omega) \right\}$$
with the norm defined by $\norm{ \vv}_V^2 = \norm{\vv}^2 + \norm{ \nabla \cdot \vv }^2.$

The variational formulation of \eqref{maineq} is defined as the following: Find $(p,\uu,\s) :[0,T] \rightarrow W\times  V \times \tilde W$ such that 
\begin{subequations}\label{weakform}
\begin{align}
&\label{W1}    ( \bar p_t, w) +  \left(\nabla\cdot \uu, w\right) =(f-\Psi_t ,w),   \quad && \forall w\in W,    \\
&\label{W2}     (\uu, \zz)  + ( K(|\s| )\s ,\zz )=0, \quad && \forall \zz\in \tilde W,\\
 &\label{W3}     (\s, \vv)  + ( \bar p , \nabla\cdot \vv )=(\nabla \Psi, \vv), \quad &&\forall  \vv\in V
\end{align}
\end{subequations} 
with $\bar{p}(x,0)=\bar{p}_0(x).$ 

\vspace{0.2cm}
{\bf Semidiscrete method.} Let $\{\mathcal T_h\}_h$ be a family of quasi-uniform triangulations of $\Omega$ with $h$ being the maximum diameter of the element.
Let $V_h$ be the Raviart-Thomas-N\'ed\'elec spaces ~\cite{Ned80, RavTho77a} of order $r\ge 0$ or Brezzi-Douglas-Marini spaces \cite{BDM85} of index $r$ over each triangulation $\mathcal T_h$, $W_h$ the space of discontinuous piecewise polynomials of degree $r$ over $\mathcal T_h$, $\tilde W_h$ the n-dimensional vector space of discontinuous piecewise polynomials of degree $r$ over $\mathcal T_h$.  Let $W_h \times V_h \times  \tilde W_h$ be the mixed element spaces approximating to $W\times V \times\tilde W$.     
  
We use the standard $L^2$-projection operator  (see \cite{Ciarlet78})  
$\pi: W \rightarrow W_h$,   $\pi: \tilde W \rightarrow \tilde W_h$ satisfying
\beqs
( \pi w , \nabla \cdot \vv_h ) = ( w , \nabla \cdot \vv_h )  
\eeqs
 for all  $w\in W, \vv_h \in V_h,$ and
\beqs 
   ( \pi \zz , \zz_h ) = ( \zz , \zz_h )
\eeqs
 for all  $\zz\in \tilde W, \zz_h\in \tilde W_h.$ 
 
Also we use  $H$-div projection $\Pi : V
\rightarrow V_h$ defined by
\beqs
( \nabla \cdot \Pi \vv , w_h )
= ( \nabla \cdot \vv , w_h )
\eeqs
for all $w_h \in W_h$.

These projections have well-known approximation properties as in~\cite{BF91, JT81}. %Below are the standard approximation properties for these projections  

(i) $\norm{\pi w}\le \norm{w}$ holds for all $w\in L^2(\Omega)$.
   
(ii) There exist positive constants $C_1, C_2$ such that
\begin{equation}
\label{prjpi}
\begin{split}
\norm{\pi w - w }_{L^\alpha(\Omega)} \leq C_1 h^m \norm{w}_{m,\alpha} \text{ and } 
\norm{ \pi \zz - \zz }_{L^\alpha(\Omega)} \leq C_2 h^m \norm{\zz}_{m,\alpha}, 
\end{split}
\end{equation}
for all $w \in W^{m,\alpha}(\Omega)$, $\zz\in (W^{m,\alpha}(\Omega))^d$,  $0\le m \le r+1, 1\le\alpha\le \infty$. Here $\norm {\cdot}_{m,\alpha}$ denotes a standard norm in Sobolev space $W^{m,\alpha}$. In short hand, when $\alpha=2$ we write \eqref{prjpi} as   
\beqs
\norm{\pi w - w } \leq C_1 h^m \norm{w}_{m}, \quad  \text { and }\quad \norm{ \pi \zz - \zz } \leq C_2 h^m \norm{\zz}_m. 
\eeqs

(iii) There exists a positive $C_3$ such that
\beq
\label{prjPi}
\norm{\Pi \vv - \vv}_{L^\alpha(\Omega)}  \leq C_3 h^m \norm{ \vv }_{m,\alpha}
\eeq
for any $\vv \in \left( W^{m,\alpha} ( \Omega ) \right)^d,$  $1/\alpha \leq m \leq r+1$, $ 1\le \alpha \le \infty $.

Because of the commuting relation between $\pi, \Pi$ and the
divergence (i.e., that $\nabla \cdot \Pi \uu = \pi ( \nabla \cdot \uu
)$, we also have the bound
\begin{equation}
\label{prjdiv}
\| \nabla \cdot ( \Pi \vv - \vv ) \|_{L^\alpha(\Omega)} \leq C_1 h^m \norm{\nabla \cdot \vv }_{m,\alpha},
\end{equation}
provided $\nabla \cdot \vv \in W^{m,\alpha}(\Omega)$ for $1 \leq m \leq r+1$.

 The semidiscrete expanded mixed formulation of~\eqref{weakform} can read as following: Find $(p_h,\uu_h,\s_h):[0,T] \rightarrow W_h\times V_h\times \tilde W_h$ such that
\begin{subequations}\label{semidiscreteform}
\begin{align}
  &\label{Semi1}  ( \bar p_{h,t}, w_h) + \left(\nabla\cdot \uu_h, w_h\right) =(f-\Psi_t,w_h),   &&\forall w_h\in W_h,    \\
  &\label{Semi2}   (\uu_h,\zz_h)  + ( K(|\s_h| )\s_h   , \zz_h )=0,  &&\forall  \zz_h\in \tilde W_h,\\
  &\label{Semi3}    (\s_h,\vv_h)  + ( \bar p_h , \nabla\cdot \vv_h )= (\nabla \Psi, \vv_h ),  &&\forall  \vv_h\in V_h,
\end{align}
\end{subequations} 
where $\bar{p}_h(x,0)=\pi \bar{p}_0(x)$ and  $\bar p_h =p_h -\pi\Psi.$  

\vspace{0.2cm}
{\bf Fully discrete method.}
We use backward Euler for time-difference discretization. Let $N$ be the positive integer, $t_0=0 < t_1 <\ldots < t_N= T$ be partition interval $[0,T]$ of $N$ sub-intervals, and let $\Delta t = t_{n} - t_{n-1}=T/N$  be the $n$-th time step size, $t_n=n\Delta t$ and $\varphi^n = \varphi(\cdot, t_n)$. 

The discrete time expanded mixed finite element approximation to \eqref{weakform} is defined as follows:  Find $ (p_h^n, \uu_h^n,\s_h^n)\in W_h\times V_h\times \tilde W_h$, $n=1,2,\dots, N$, such that 
\begin{subequations}\label{fullydiscreteform}
\begin{align}
&\label{fully1}    \left( \frac{\bar p_h^n - \bar p_h^{n-1}}{\Delta t }, w_h\right) +  \left(\nabla\cdot \uu_h^n, w_h\right) =(f^n-\Psi_t^n, w_h ), &&\forall w_h\in W_h, \\
&\label{fully2}     (\uu_h^n, \zz_h)  + ( K(|\s_h^n| )\s_h^n ,\zz_h )=0, &&\forall \zz_h\in \tilde W_h,\\
 &\label{fully3}     (\s_h^n,{\bf v_h})  + ( \bar p_h^n , \nabla\cdot  \vv_h )=(\nabla \Psi^n, \vv_h ), &&\forall \vv_h\in W_h.
\end{align}
\end{subequations} 
The initial approximations are chosen as follows:
$$
\bar p_h^0(x)=\pi \bar p_0(x),\quad  \s_h^0(x)=\pi \nabla p^0(x), \quad  \uu_h^0(x) =K(|\s_h^0(x)|)\s_h^0(x) 
$$
for all $x\in \Omega$.
%------------------------------------------------------------
\section{Estimates  of solutions } \label{prior-est}
Using the theory of monotone operators \cite{MR0259693,s97,z90}, the authors in \cite {HIKS1} proved the global existence of weak solution $p(x,t)$ of equation \eqref{maineq}. Furthermore this solution is unique and  belongs to $C([0,T),L^\alpha(\Omega))$, $\alpha\ge 1$ and $L_{loc}^{\beta}([0,T),W^{1,\beta}(\Omega)),$ $p_t \in L^{\beta'}_{loc}\bigl([0,T), (W^{1,\beta}(\Omega))'\bigr)\cap  L^2_{loc}\bigl([0,T), L^2(\Omega)\bigr)$
  provided the initial data $p_0(x) \in L^2(\Omega)\cap W^{1,\beta}(\Omega)$ and $\Psi,$ $f$  sufficiently smooth. In fact, in \cite{HKP1, HKP2} the authors show that $ p(x; t)\in L^\infty((0,T);L^\infty(\Omega))\cap L^\infty((0, T);W^{1,\beta}(\Omega))$ and  $p_t(x; t)\in L^\infty_{loc}((0, T);L^2(\Omega))$. Our aim explores the properties of the solutions, we assume that $p(x; t)$, initial data and boundary data have sufficiently regularities both in $x$ and $t$ variables so that our calculations are valid. 

We begin with the Poincar\'e-Sobolev inequality with a specific weight which is essential in our estimate later.      
 \begin{lemma}[cf. \cite{HI2}] Let $\Omega$ be an open bounded domain in $\mathbb{R}^d$ and $\xi(x)\ge 0$ be defined on $\Omega$. Then for any function $u(x)$ vanishing on the boundary $\partial \Omega$ there is a positive constant $C$ depending of $\Omega,$ $\deg(g)$ and coefficients of $g$ such that.
\beq\label{PSineq}
\norm{u}_{L^{\beta^*}}^2 \le C \norm{K^{\frac 12}(\xi) \nabla u}^2\left(1+ \norm{K^{\frac 12}(\xi) \xi}^2\right)^{\gamma},
\eeq
where $ \beta^* = \frac{d\beta}{d-\beta}.$
\end{lemma} 
%----------------------------------------------------------
\begin{theorem}\label{ph}  Let $(p, \uu,\s)$ solve problem \eqref{weakform}. 

{\rm (i)} There is a positive constant $C$ such that for any $t\in (0,T)$,  
  \beq\label{Bph}
\norm{\bar p(t)}^2 + \int_0^t \norm{ K^{\frac 12}(|\s(\tau)| )\s(\tau)}^2d\tau \le \norm{\bar p^0}^2 +C \int_0^t A(\tau)d\tau, 
\eeq   
where 
$$
A=A(t)=\norm{\nabla \Psi(t)}^2+ \norm{(f-\Psi_t)(t)}_{L^r(\Omega)}+  \norm{(f-\Psi_t)(t)}_{L^r(\Omega)}^{\frac {\beta}{\beta-1}}
$$ 
with $r=\frac{d\beta}{\beta(d+1)-d}.$
Consequently, 
 \beqs
\norm{p(t)}^2 + \int_0^t \norm{ K^{\frac 12}(|\s(\tau)| )\s(\tau)}^2d\tau \le \norm{\bar p^0}^2 + \norm{\Psi}^2 +C \int_0^t A(\tau)d\tau. 
\eeqs 
  
{\rm (ii)} There exist positive constants $C, C_1$ such that for any $t\in (0,T)$,
\beq\label{suh}
\begin{split}
 \norm{\uu(t)}^2+ \norm{\s(t)}_{L^\beta(\Omega)}^{\beta} \le C\left(\norm{\bar p^0}^2+1\right)+ C\int_0^t e^{- C_1(t-\tau)}(\Lambda+B)(\tau)d\tau,
\end{split}
\eeq
where 
\beq\label{lda}
\Lambda=\Lambda(t) =\int_0^t A(\tau)d\tau,
\eeq
\beq\label{g3}
B=B(t)=A(t)+\norm{\nabla \Psi_t(t)}^2 +\norm{(\Psi_t -f)(t)}^2.
\eeq
\end{theorem}
%==========================
\begin{proof}
(i) By selecting $w=\bar p$, $\zz=\s$ and $\vv=\uu$ at each time level in \eqref{weakform}  we have 
\begin{align*}
  &( \bar p_t, \bar p) + \left(\nabla\cdot \uu, \bar p\right) =(f-\Psi_t , \bar p), \\
  &(\uu,\s)  + ( K(|\s| )\s   , \s )=0, \\
  &(\s,\uu)  + ( \bar p , \nabla\cdot \uu )= (\nabla \Psi, \uu ).
\end{align*}
Adding three above equations implies 
\beq\label{Ineq}
\begin{aligned}
\frac 12 \frac{d}{dt} \norm{\bar p}^2 + \norm{ K^{\frac 12}(|\s| )\s}^2 &= (f-\Psi_t , \bar p) - (\nabla \Psi, \uu )\\
&\le(f-\Psi_t, \bar p) +\frac 1 2(\norm { \nabla \Psi}^2 +\norm{ \uu }^2).
\end{aligned}
\eeq
Using \eqref{W2} with $\zz=\uu \in \tilde W $, we have 
 \beqs
   \norm{\uu}^2 =-  ( K(|\s| )\s   , \uu )\le \norm{K^{\frac 12}(|\s| )\s}\norm{\uu},
   \eeqs
 which yields 
 \beq\label{uuh}
 \norm{\uu}\le \norm{K^{\frac 12}(|\s| )\s}.
 \eeq  
Thus \eqref{Ineq} and \eqref{uuh} give   
\beq\label{mainineq}
\frac{d}{dt} \norm{\bar p}^2 + \norm{ K^{\frac 12}(|\s| )\s}^2 \le 2(f-\Psi_t , \bar p) +\norm { \nabla \Psi}^2. 
\eeq  
 By H\"oder's inequality and \eqref{PSineq},	  
\beq\label{d1}
\begin{split}
( f-\Psi_t ,\bar p  ) &\le \norm{f-\Psi_t}_{L^r(\Omega)}\norm{\bar p}_{L^{\beta^*}}\\
      &\le C \norm{f-\Psi_t}_{L^r(\Omega)}\norm{K^{\frac 12}(|\s|)\nabla \bar  p}\left( 1+\norm{K^{\frac 12}(|\s|)\s} \right)^{\gamma}.
\end{split}
\eeq
To estimate the second term on the right hand side of \eqref{d1} we integrate by part \eqref{W3} and then select $\vv =K(|\s|)\nabla \bar  p \in V$. It follows that
\beqs
( \nabla \bar p , K(|\s|)\nabla \bar  p )=(\s,K(|\s|)\nabla \bar  p)  - (\nabla \Psi, K(|\s|)\nabla \bar  p )
\eeqs
which shows that  
\begin{align*}
  \norm{K^{\frac 12}(|\s|)\nabla \bar  p}^2&=(\s-\nabla \Psi,K(|\s|)\nabla \bar  p)\le \norm {K^{\frac 12}(|\s|)(\s-\nabla \Psi)}\norm{K^{\frac 12}(|\s|)\nabla \bar  p}.
\end{align*}
This, triangle inequality and the upper boundedness of $K(\cdot)$ give 
\beq\label{weightnorm}
\norm{K^{\frac 12}(|\s|)\nabla \bar  p}\le \norm {K^{\frac 12}(|\s|)\s} +C\norm{\nabla \Psi}.
\eeq
Combining \eqref{d1}, \eqref{weightnorm} and Young's inequality yield
\beq\label{d2}
\begin{aligned}
(f -\Psi_t ,\bar p  ) &\le C \norm{f-\Psi_t}_{L^r(\Omega)}\left(\norm {K^{\frac 12}(|\s|)\s} +\norm{\nabla \Psi}\right)\left( 1+\norm{K^{\frac 12}(|\s|)\s} \right)^{\gamma}\\
&\le C\norm{f-\Psi_t}_{L^r(\Omega)}\left(1+\norm {K^{\frac 12}(|\s|)\s}^{\gamma+1} +\norm{\nabla \Psi}^{\gamma+1}\right)\\
%&= C\norm{f-\Psi_t}_{L^r(\Omega)} +  C\norm{f-\Psi_t}_{L^r(\Omega)} \Big(\norm {K^{\frac 12}(|\s|)\s}^{\gamma+1} +\norm{\nabla \Psi}^{\gamma+1} \Big)\\
&\le \frac 12 \norm{K^{\frac 12}(|\s|)\s}^2 + C A(t).
\end{aligned}
\eeq
It follows from \eqref{mainineq} and \eqref{d2} that    
 \beq\label{dervp}
\frac{d}{dt} \norm{\bar p}^2 + \norm{ K^{\frac 12}(|\s| )\s}^2 \le C A(t). 
\eeq  
Integrating \eqref{dervp} from $0$ to $t$ we obtain \eqref{Bph}.

(ii) Choosing $w=\bar p_t$, $\zz=\s_t$ in \eqref{W1}, \eqref{W2}, differentiating  \eqref{W3} with respect $t$ and selecting $\vv=\uu$ we find that  
\begin{align*}
  &( \bar p_t, \bar p_t) + \left(\nabla \cdot \uu, \bar p_t\right) =(f-\Psi_t , \bar p_t), \\
  &(\uu,\s_t)  + ( K(|\s| )\s   , \s_t )=0, \\
  &(\s_t,\uu)  + (\bar p_t ,  \nabla \cdot \uu )= (\nabla \Psi_t, \uu ).
\end{align*}
Adding three resultant equations gives   
\beq\label{ptEq}
\norm{ \bar p_t }^2 + ( K(|\s| )\s   , \s_t ) =(f-\Psi_t , \bar p_t)- (\nabla \Psi_t, \uu ).
\eeq 
Using \eqref{W2} with $\zz =\nabla \Psi_t\in \tilde W$ we have 
$$
(\uu, \nabla \Psi_t ) = -( K(|\s| )\s   ,\nabla \Psi_t )\le \norm{K^{\frac 12}(|\s|)\s} \norm{\nabla \Psi_t }.
$$
Note that definition of function $H(\cdot)$ in \eqref{Hdef} gives      
$$ K(|\s| )\s  \cdot \s_t =\frac12\frac d{dt} H(\s).$$
Thus, \eqref{ptEq} yields 
\beq\label{Bbarp}
\begin{aligned}
\norm{ \bar p_t }^2 + \frac12 \frac d {dt}\int_\Omega H(x,t)dx \le C\norm{K^{\frac 12}(|\s|)\s} \norm{\nabla \Psi_t }+\norm{f-\Psi_t}\norm{ \bar p_t}&\\
  \le \varep \norm{K^{\frac 12}(|\s|)\s}^2  + C_\varep\norm{\nabla \Psi_t}^2+\frac 1 2\left(\norm{f-\Psi_t}^2+\norm{ \bar p_t}^2\right)& 
\end{aligned}
\eeq
for all $\varep>0$, where $H(x,t)=H(\s(x,t)).$ 

Adding \eqref{dervp} and \eqref{Bbarp} and selecting sufficiently small $\varep$ implies 
\beqs
\norm{ \bar p_t }^2 + \frac d {dt}\int_\Omega H(x,t)dx+(\bar p, \bar p_t) +  C \norm{K^{\frac 12}(|\s|)\s}^2\le  C B(t).
\eeqs 
Then by Cauchy's inequality:  
\beqs
\frac 12\norm{ \bar p_t }^2 + \frac d {dt}\int_\Omega H(x,t)dx \le - C \norm{K^{\frac 12}(|\s|)\s}^2+ \frac 12\norm{\bar p}^2  +C B(t).
\eeqs 
This and \eqref{i:ineq4} show that 
\beqs
\frac 12\norm{ \bar p_t }^2 + \frac d {dt}\int_\Omega H(x,t)dx \le - C\int_\Omega H(x,t)dx + \frac 12\norm{\bar p}^2  +C B(t).
\eeqs
Ignoring the first term and using Gronwall's inequality we obtain   
\beq\label{Avept}
\begin{split}
\int_\Omega H(x,t)dx &\le -e^{- C_1t}\int_\Omega H(x,0)dx\\
&\quad +C \int_0^t e^{-C_1(t-\tau)}\big(\norm{\bar p}^2  +B\big)d\tau.
\end{split}
\eeq
Using \eqref{i:ineq2}, \eqref{i:ineq4} and \eqref{Bph}, we have from \eqref{Avept} that  
\beq\label{ssh}
\begin{split}
 \norm{\s}_{L^\beta(\Omega)}^{\beta}  &\le -e^{- C_1 t}\norm{\s^0}_{L^{\beta}(\Omega) }^{\beta} + C+ C\int_0^t e^{-C_1(t-\tau)}\left(\norm{\bar p^0}^2 + \Lambda +B\right)d\tau\\
&\le C\left( \norm{\bar p^0}^2+ 1\right)+C\int_0^t e^{-C_1(t-\tau)}\left(\Lambda +B\right)d\tau .
\end{split}
\eeq
In addition, due to \eqref{uuh} and $K(\xi)\xi^2\le C\xi^{2-a}=C\xi^\beta,$
\beq\label{uh}
 \norm{\uu}^2\le \norm{\s}^{\beta}_{L^{\beta}(\Omega)}.
 \eeq 
Combining \eqref{uh} and \eqref{ssh},  we obtain \eqref{suh}. The proof is complete. 
\end{proof}
%----------------------------------------

Although solution is considered continuous at $t=0$ in appropriate Lebesgue or Sobolev space. Its time derivative is not. In the following we prove the time derivative of pressure is bounded. 
\begin{theorem}\label{phderv} There is a positive constant $C$ such that for any $0<t_0\le t\le T,$
\beq\label{Bpht}
\begin{split}
\norm{ \bar p_t(t)}^2&\le C\Big\{t_0^{-1}\left(\norm{\bar p^0}^2  +\int_0^{t_0}(\Lambda +B)(\tau)d\tau\right)\\
 &\quad+ \int_0^t (\norm {(f_t-\Psi_{tt})(\tau)}^2 + \norm{\nabla \Psi_t(\tau)}^2)d\tau \Big\},
\end{split}
\eeq
where $B(t)$ is given as in \eqref{g3}.
%In particular,  if solution $\bar p_h$ has time derivative continuous at $0$ then for all $t\in (0,T)$
%\beq\label{Smoothpht}
%\norm{ \bar p_{h,t}(t)}^2\le C \Big\{\norm{\nabla\cdot \uu_h^0}^2+\norm{f^0-\Psi_t^0}^2 + \int_0^t (\norm {f_t(\tau)-\Psi_{tt}(\tau)}^2 + \norm{\nabla \Psi_t(\tau)}^2)d\tau\Big\}.
%\eeq
\end{theorem}
\begin{proof}
We differentiate \eqref{weakform} with respect $t$, 
\begin{subequations}
\begin{align}
&\label{ds1}    ( \bar p_{tt}, w) +  \left(\nabla\cdot \uu_{t}, w\right) =(f_t-\Psi_{tt}, w), &&\forall w\in W, \\
&\label{ds2}     (\uu_{t}, \zz)  +  ( K(|\s|)\s_{t}, \zz) +\left(K'(|\s|) \frac{\s\cdot \s_{t}}{|\s|}\s,\zz \right)=0, &&\forall \zz\in \tilde W, \\
 &\label{ds3}     (\s_{t},\vv)  + ( \bar p_{t} , \nabla\cdot \vv )= (\nabla \Psi_t, \vv ), &&\forall \vv\in V.
\end{align}
\end{subequations}
Selecting $w=\bar p_{t}$, $\zz=\s_{t}$ and $\vv =\uu_{t}$ and summing resultant equations we obtain 
\beq\label{I}
\begin{split}
\frac12\frac d{dt}  \norm{ \bar p_{t}}^2+ \norm{K^{\frac 12}(|\s|)\s_{t}}^2 &= - \left(K'(|\s|) \frac{\s\cdot \s_{t}}{|\s|}\s,\s_{t} \right)+(f_t-\Psi_{tt},\bar p_{t})\\
 &\quad-(\nabla \Psi_t, \uu_{t} )=I_1+I_2+I_3.
\end{split}
\eeq
According to \eqref{i:ineq3},   
\beqs
\left| K'(|\s|) \frac{\s\cdot \s_{t}}{|\s|}\s\right| \le a K(|\s|)|\s_{t}|,
\eeqs
which leads to
\beq\label{I1}
|I_1|\le  a\norm{K^{\frac 12}(|\s|) \s_{t}}^2.
\eeq
By Cauchy's inequality, 
\beq\label{I2}
|I_2| \le \frac12 \left(\norm {f_t-\Psi_{tt}}^2+ \norm{\bar p_{t}}^2 \right).
\eeq
For all $\varep>0$,
\beqs
|I_3|\le C\varep^{-1}\norm{\nabla \Psi_t}^2+\varep\norm{\uu_{t}}^2. 
\eeqs
In \eqref{ds2}, taking $\zz=\uu_{t}$ we find that  
\beq\label{B4u}
\norm{\uu_{t}}\le   \norm{ K^{\frac 12}(|\s|)\s_{t}} +\norm{K'(|\s|) \frac{\s\cdot \s_{t}}{|\s|}\s}\le (1+a)\norm{ K^{\frac 12}(|\s|)\s_{t}}.
\eeq
By choosing $\varep=\frac{1-a}{2(1+a)}>0$, 
\beq\label{I3}
|I_3|\le C\norm{\nabla \Psi_t}^2+\frac{1-a}2\norm{ K^{\frac 12}(|\s_h|)\s_{t}}^2. 
\eeq
 It follows from \eqref{I}, \eqref{I1}, \eqref{I2} and \eqref{I3} that
\beqs
\frac d{dt}  \norm{ p_{t}}^2+ (1-a) \norm{K^{\frac 12}(|\s|)\s_{t}}^2 \le \norm {f_t-\Psi_{tt}}^2+ \norm{\bar p_{t}}^2 + C\norm{\nabla \Psi_t}^2.
\eeqs
Dropping the nonnegative term of the left hand side gives  
 \beq\label{DEpht}
\frac d{dt}  \norm{ \bar p_{t}}^2 \le   \norm{\bar p_{t}}^2+C(\norm {f_t-\Psi_{tt}}^2 + \norm{\nabla \Psi_t}^2).
\eeq
For $t\ge t'>0$, applying Gronwall's inequality to \eqref{DEpht} we find  that
\beq\label{phtB}
\norm{ \bar p_{t}}^2\le C\left(\norm{ \bar p_{t}(t')}^2 + \int_{t'}^t (\norm {f_t-\Psi_{tt}}^2 + \norm{\nabla \Psi_t}^2)d\tau\right).
\eeq
Integrating \eqref{phtB} in $t'$ from $0$ to $t_0$, using \eqref{Avept} we obtain 
\begin{multline*}
t_0\norm{ \bar p_{t}}^2\le C\left(\int_0^{t_0} \norm{ \bar p_{t}(t')}^2 + \int_0^{t_0}\int_{t'}^t (\norm {\Psi_{tt}-f_t}^2 + \norm{\nabla \Psi_t}^2)d\tau\right)\\
\le C\left(\int_0^{t_0} e^{- C_1(t_0-\tau)}\left(\norm{\bar p}^2  +B\right)d\tau + t_0 \int_0^t (\norm {f_t-\Psi_{tt}}^2 + \norm{\nabla \Psi_t}^2)d\tau \right).
\end{multline*}
Now using \eqref{Bph} we deduce previous inequality to   
\beqs
t_0\norm{ \bar p_{t}}^2\le C\left(\norm{\bar p^0}^2 + \int_0^{t_0}(\Lambda+B)d\tau + t_0 \int_0^t (\norm {f_t-\Psi_{tt}}^2 + \norm{\nabla \Psi_t}^2)d\tau \right)
\eeqs
which proves \eqref{Bpht}. 
The proof is complete.  
\end{proof}
%---------------------------------------------------------------

We also obtain the similar results for solution of \eqref{semidiscreteform} as following.    
\begin{theorem}\label{pspt} Let $(p_h, \uu_h,\s_h)$ solve the semidiscrete problem \eqref{semidiscreteform}. 

{\rm (i)}   There is a positive constant $C$ such that for each $t\in (0,T),$
\beq\label{Bp}
\norm{\bar p_h(t)}^2 + \int_0^t \norm{ K^{\frac 12}(|\s_h(\tau)| )\s_h(\tau)}^2d\tau \le \norm{\bar p^0}^2 +C \int_0^t A(\tau)d\tau. 
\eeq   

{\rm (ii)}   There are two positive constants $C,  C_1$ such that for each $t\in (0,T)$   
\beq\label{su}
\begin{split}
\norm{\uu_h(t)}^2 +\norm{\s_h(t)}_{L^{\beta}(\Omega)}^{\beta}\le C\left(\norm{\bar p^0}^2+1+ \int_0^t e^{-C_1(t-\tau)}(\Lambda +B)(\tau)d\tau\right).
\end{split}
\eeq

{\rm (iii)} 
For $0<t_0\le t\le T$,  
\beq\label{Bpt}
\begin{aligned}
\norm{ \bar p_{h,t}}^2&\le C\left\{t_0^{-1}\left(\norm{\bar p^0}^2  +\int_0^{t_0}(\Lambda+B)(\tau)d\tau\right)\right.\\
 &\quad\quad \left.+ \int_0^t (\norm {(f_t-\Psi_{tt})(\tau)}^2 + \norm{\nabla \Psi_t(\tau)}^2)d\tau \right\},
\end{aligned}
\eeq
where $C$ is a positive constant.   
\end{theorem}
%==========================

To finish this section we give a bound of pressure in $L^\infty$-norm which is useful for our error estimate in the later sections. 
\begin{theorem}\label{pinf} Let $(p, \uu,\s)$ solve the problem \eqref{weakform}, $
\mu \ge 2 \text{ and } \mu>\frac{ad}{\beta}.
$
If $T>0$ then 
\beq\label{barpone}
\begin{split}
&\sup_{t\in[0,T]}\norm{\bar p(t)}_{L^\infty(\Omega)}\le 2\|\bar p^0\|_{L^\infty(\Omega)}\\
&+ C\left\{ (1+T)^{\mu} \Big(1+ \sup_{t\in[0,T]} \|(f-\Psi_t)(t)\|_{L^{\mu+1}(\Omega)}^\mu+\sup_{t\in[0,T]}\|\nabla \Psi(t)\|_{L^\infty(\Omega)}^{\frac{\beta\mu}{2}}\Big)\right\}^\frac1{\mu-a}.
\end{split}
\eeq
\end{theorem}
The proof of Theorem~\ref{pinf} is given in Appendix.  
%--------------------------------------------------

\section{Error analysis}\label{Err-anl} 
    
In this section, we use estimates in the previous section, the techniques in \cite{HI1} and the expanded mixed finite element method to establish the error estimates between the analytical solution and approximation solution in several norms. 

%============================
%In the below development we discuss error estimates for the case $K(\cdot)$ is unbounded below. In doing so we use the boundedness of gradient of pressures in $L^\beta(\Omega)(\Omega)$ which are given in Theorems~\ref{ph} and \ref{pspt}. For this purpose, we assume the weak solutions, initial data and boundary data enough regularities so that our calculation can be applied.         

%============================
\subsection{Error estimate for semidiscrete method} 
We will bound the error in the semidiscrete method in various norms by
comparing the computed solution to the projections of the true
solutions.  To do this, we restrict the test functions
in~\eqref{weakform} to the finite-dimensional spaces.   
Let 
\begin{align*}
 \bar p_h -\bar p =(\bar p_h-\pi \bar p ) + (\pi \bar p - \bar p) \equiv  \vartheta + \theta,\\
  \s_h -\s = (\s_h-\pi \s ) + (\pi \s - \s) \equiv  \eta+ \zeta,\\
  \uu_h -\uu = (\uu_h-\Pi \uu ) + (\Pi \uu - \uu) \equiv  \rho+ \varrho.
\end{align*}
Properties of projections in \eqref{prjpi} and \eqref{prjPi} yield for each $t\in[0,T],$
\begin{align}
\label{Btheta}
\norm{\theta }_{L^\alpha} &\le Ch^m \norm{\bar p }_{m,\alpha},\quad  \forall \bar p\in W^{m,\alpha}(\Omega),\\
\label{Bzeta}
\norm{\zeta }_{L^\alpha} &\le Ch^m \norm{\s}_{m,\alpha},   \quad\forall \s\in  (W^{m,\alpha}(\Omega))^d,\\
\label{Bvarrho}
\norm{\varrho }_{L^\alpha}&\le Ch^m \norm{\uu }_{m,\alpha},\quad \forall \uu\in (W^{m,\alpha}(\Omega))^d. 
\end{align}  
for all $1\le m\le r+1$, $1\le \alpha\le \infty.$ Let $t_0>0$, 
\begin{align*}
\Upsilon &= 1+\norm{\bar p^0}^2 + \sup_{t\in [0,T]}\int_0^t e^{-C_1(t-\tau)}(\Lambda+B)(\tau)d\tau, \\
\Xi &=t_0^{-1}\left(\norm{\bar p^0}^2  +\int_0^{t_0}(\Lambda+B)(\tau)d\tau\right) + \int_0^T (\norm {(f_t-\Psi_{tt})(\tau)}^2 + \norm{\nabla \Psi_t(\tau)}^2)d\tau.
\end{align*}
where $\Lambda(t)$ and $B(t)$ are defined in \eqref{lda}, \eqref{g3}.

\begin{theorem}\label{mainres} Assume $(\bar p^0,\uu^0,\s^0 )\in W\times V\times \tilde W$ and $(\bar p_h^0,\uu^0_h,\s_h^0 )\in W_h\times V_h\times \tilde W_h$. Let $(p, \uu,\s)$ solve problem \eqref{weakform} and $(p_h, \uu_h,\s_h)$ solve the semidiscrete mixed finite element approximation \eqref{semidiscreteform}. There is a positive constant $C$ such that for any $t\in (0,T)$,
\beq\label{mi2a}
\norm {(p_h - p)(t)}^2 \le C\left( \norm {\theta(t)}^2+ \norm{ (\pi \Psi -\Psi)(t) }^2\right)+ C  \Upsilon \int_0^t   \norm{\zeta(\tau)}_{L^\beta(\Omega)}d\tau.
\eeq
Consequently, if $p, \Psi\in L^\infty(0,T; H^{r+1}(\Omega)) $, $\s \in L^2(0,T; (W^{r+1,\beta}(\Omega))^d)$ then for any $t\in (0,T)$, 
\beq \label{mi4a}
\begin{aligned}
\norm {(p_h - p)(t)  }&\le Ch^{r+1} \left( \norm {\bar p(t)}_{r+1}+\norm{\Psi(t)}_{r+1}\right)\\
 &\quad+  C\Upsilon^{\frac 12}h^{\frac{r+1}2}\sqrt{\int_0^t \norm{\s(\tau)}_{r+1,\beta} d\tau}. 
\end{aligned}
\eeq
\end{theorem}
%=======================
\begin{proof} 
Subtracting \eqref{semidiscreteform} from \eqref{weakform} we obtain the equations of difference:   
\begin{subequations}\label{Erreq}
\begin{align}
&\label{Erreq1} ( \bar p_{h,t} -\bar p_t , w_h) +  \left(\nabla\cdot (\uu_h -\uu) , w_h\right) =0,    &&\forall w_h\in W_h,\\
&\label{Erreq2} (\uu_h - \uu,\zz_h)  + \left( K(|\s_h| )\s_h - K(|\s|) \s, \zz_h \right)=0,  &&\forall \zz_h\in \tilde W_h,\\
&\label{Erreq3}  (\s_h-\s,\vv_h)  + ( \bar p_h - \bar p , \nabla\cdot \vv_h )= 0, &&\forall \vv_h\in V_h.
\end{align}
\end{subequations} 
%where initial condition is defined by $\bar{p}_h(x,0)=\pi \bar{p}_0(x)$  in $\Omega$.

Let $w_h=\vartheta,$ $\zz_h=\eta $ and $\vv_h=\rho$. Using the $L^2$-projection and $H(\rm div)$-projection, we have from \eqref{Erreq1}--\eqref{Erreq3} that  
\begin{subequations}\label{eqprj}
\begin{align}
\label{prj1} &( \vartheta_t, \vartheta) +  \left(\nabla\cdot \rho, \vartheta \right) =0, \\
\label{prj2}&(\rho, \eta)  + \left( K(|\s_h| )\s_h -K(|\s|)\s  ,  \eta\right)=0,\\ 
\label{prj3} &(\eta,\rho)  + ( \vartheta , \nabla\cdot \rho )= 0.
\end{align}
\end{subequations} 
%where $\vartheta(0)=0.$ 
Adding three equations \eqref{prj1}--\eqref{prj3} gives  
\beqs
\frac 12  \frac d{dt}\norm {\vartheta}^2 + \left( K(|\s_h| )\s_h -K(|\s|) \s ,  \eta \right) =0
\eeqs
or 
\beq\label{eqerr}
\frac 12  \frac d{dt}\norm {\vartheta}^2 + \left( K(|\s_h| )\s_h - K(|\s|) \s ,  \s_h - \s \right) =\left( K(|\s_h| )\s_h - K(|\s|) \s , \zeta \right).
\eeq
Applying \eqref{Mono} to the second term of \eqref{eqerr} we have    
\beq\label{Mn1}
\left( K(|\s_h| )\s_h - K(|\s|) \s, \s_h -\s \right)\ge C\omega \norm{\s_h-\s}^2_{L^\beta(\Omega)},
\eeq
where
\beq\label{omegadef}
\omega = \omega(t) = \left(1+ \max\{\norm{\s_h(t)}_{L^\beta(\Omega)} ,  \norm{\s(t)}_{L^\beta(\Omega)} \}  \right)^{-a}.
\eeq
Since $K(|\xi| )\xi \le C\xi^{\beta-1} $, the last term of \eqref{eqerr} is bounded by 
\beq\label{mi1a}
\begin{aligned}
\left| ( K(|\s_h| )\s_h - K(|\s|) \s, \zeta) \right|&\le C\left(|\s_h|^{\beta-1} +|\s|^{\beta-1},|\zeta|\right)\\
&\le C \left(\norm{\s_h}_{L^\beta(\Omega)}^{\beta-1}+\norm{\s}_{L^\beta(\Omega)}^{\beta-1} \right)  \norm{\zeta}_{L^\beta(\Omega)}.
\end{aligned}
\eeq
The last inequality is obtained by applying H\"oder's inequality with powers $\frac {\beta}{\beta-1}$ and $\beta$. 

It follows from \eqref{eqerr}, \eqref{Mn1} and \eqref{mi1a} that
\beq\label{keya}
\begin{split}
\frac d{dt}\norm {\vartheta}^2 +  \omega\norm{\s_h -\s}_{L^\beta(\Omega)}^2 
&\le \left(\norm{\s_h}_{L^\beta(\Omega)}^{\beta-1}+\norm{\s}_{L^\beta(\Omega)}^{\beta-1} \right) \norm{\zeta}_{L^\beta(\Omega)}\\
&\le C\left(1+ \norm{\s_h}_{L^\beta(\Omega)}^{\beta} +  \norm{\s}_{L^\beta(\Omega)}^{\beta}  \right) \norm{\zeta}_{L^\beta(\Omega)}.
\end{split}
\eeq
Due to \eqref{suh} and \eqref{su},   
\beq\label{Bomega}
1+ \norm{\s_h}_{L^\beta(\Omega)}^{\beta} +  \norm{\s}_{L^\beta(\Omega)}^{\beta} \le C\Upsilon.
\eeq
Integrating \eqref{keya} in time, using $\vartheta(0)=0$,  we have
\beq\label{keyb}
\norm {\vartheta}^2 +\int_0^t  \omega\norm{\s_h -\s}^2_{L^\beta(\Omega)} d\tau\le C\Upsilon \int_0^t  \norm{\zeta}_{L^\beta(\Omega)}d\tau.
\eeq 
Since
\beq\label{trans1}
 p_h - p = (\bar p_h - \bar p )+ (\pi \Psi -\Psi ) = \vartheta +\theta + (\pi \Psi -\Psi ),
\eeq
 the inequality \eqref{mi2a} follows from \eqref{trans1}, Minkowski's inequality and \eqref{keyb}. 
 
 Applying \eqref{Bzeta}, \eqref{Bvarrho} to \eqref{mi2a} we obtain \eqref{mi4a}.
 The proof is complete.    
\end{proof}
%----------------------------------------------------------

The $L^2$-error estimate and the inverse estimates enable us to find the $L^\infty$- error estimate as following: 
%-----------------------------------------------------------
\begin{theorem}\label{pinferr}  Assume $(\bar p^0,\uu^0,\s^0 )\in W\times V\times \tilde W$ and $(\bar p_h^0,\uu^0_h,\s_h^0 )\in W_h\times V_h\times \tilde W_h$. Let $(p, \uu,\s)$ solve problem \eqref{weakform} and $(p_h, \uu_h,\s_h)$ solve the semidiscrete mixed finite element approximation \eqref{semidiscreteform}.  If $p, \Psi\in L^\infty\left(0,T; W^{r+1,\infty}(\Omega) \right),$  then there is a positive constant $C$ such that for any $t\in (0,T)$,
\beq\label{pphdiff}
\begin{split}
\norm{(p-p_h)(t)}_{L^\infty(\Omega)} &\le \norm{\theta(t)}_{L^\infty(\Omega)}+ \norm{(\pi\Psi -\Psi)(t)}_{L^\infty(\Omega)}\\
&\quad +C\Upsilon^{\frac 12} h^{-1} \sqrt{\int_0^t  \norm{\zeta(\tau)}_{L^\beta(\Omega)} d\tau}.  
\end{split}
\eeq 
Furthermore if $\s \in L^1\left(0,T; (W^{r+1,\beta}(\Omega))^d\right)$ then 
\beq\label{errlinf}
\begin{split}
\norm{(p-p_h)(t)}_{L^\infty(\Omega)} &\le Ch^{r+1}\Big(\norm{\bar p(t)}_{r+1, \infty}+\norm{\Psi(t)}_{r+1, \infty}\Big)\\
&\quad + C\Upsilon^{\frac 12} h^{\frac{r-1}2}\sqrt{\int_0^t  \norm{\s(\tau)}_{r+1,\beta}d\tau}.
\end{split}
\eeq
\end{theorem}
%----------------------------------------------
\begin{proof}  
We have from \eqref{trans1} and triangle inequality that 
\beq\label{tri1}
\begin{split}
\norm{p-p_h}_{L^\infty} \le \norm{\theta }_{L^\infty} +\norm{\vartheta}_{L^\infty}+ \norm{\pi\Psi -\Psi}_{L^\infty}.  
\end{split}
\eeq  
Due to the quasi-uniformly of $\mathcal T_h$, the following inverse estimate holds 
\beqs
\norm{\vartheta}_{L^\infty} \le Ch^{-\frac2 q}\norm{ \vartheta}_{L^q} \quad \text { for all } 1\le q\le\infty.  
\eeqs
 Applying this with $q=2$ and using \eqref{keyb} imply 
 \beq\label{varthinf}
 \begin{split}
 \norm{\vartheta}_{L^\infty(\Omega)} \le Ch^{-1}\norm{ \vartheta} 
 &\le C\Upsilon^{\frac 12} h^{-1} \left(\int_0^t  \norm{\zeta}_{L^\beta(\Omega)}\right)^{\frac 12}\\
 &\le C\Upsilon ^{\frac 12} h^{-1} \left(\int_0^t  \norm{\zeta}_{L^\beta(\Omega)}\right)^{\frac 12}.
 \end{split}
 \eeq 
Hence \eqref{pphdiff} follows from \eqref{tri1} and \eqref{varthinf}.

Using \eqref{pphdiff}, \eqref{Btheta} and \eqref{Bzeta} we obtain \eqref{errlinf}. 
\end{proof}
%============================
% H^{-1} error estimate
%============================
%
Now we give the bound of $\norm{ p- p_h}_{H^{-1}}$ defined by
\beqs
\norm{\cdot}_{H^{-1}(\Omega)} =\sup_{\varphi\in H_0^1(\Omega)} \frac {(\cdot, \varphi)}{\norm{\varphi}_{H^1(\Omega)} }.
\eeqs
\begin{lemma} Under the assumption of Theorem~\ref{mainres} we have 
\beq\label{diverr}
\begin{split}
\norm{\int_0^t\nabla\cdot \left(\uu_h -\uu\right)(\tau) d\tau} &\le C\Upsilon^{\frac 12} \sqrt{\int_0^t \norm{\zeta(\tau)}_{L^\beta(\Omega)}d\tau}\\
&\quad+ \norm{\int_0^t\nabla\cdot \left(\uu -\Pi\uu\right)(\tau) d\tau}.
\end{split}
\eeq
\end{lemma}
%---------------------------------------------------
\begin{proof}
The $L^2$-projection allows us to write $(\bar p_h -\bar p , w_h)=( \vartheta , w_h)$. 

Integrating \eqref{Erreq1} from $0$ to $t$ we have 
\beq\label{divu}
( \vartheta , w_h) +   \left(\int_0^t\nabla\cdot \rho d\tau , w_h\right)  =( \vartheta(0) , w_h)=0 
\eeq
for all $w_h\in W_h$. Now choose $w_h = \int_0^t\nabla\cdot \rho d\tau\in W_h$ then

\begin{align*}
\norm{\int_0^t\nabla\cdot \rho d\tau} \le \norm {\vartheta}.
\end{align*}
Triangle inequality yields
\begin{align*}
\norm{\int_0^t\nabla\cdot (\uu_h -\uu) d\tau} \le \norm {\vartheta}+ \norm{\int_0^t\nabla\cdot \varrho d\tau}.
\end{align*}
Using \eqref{keyb} we obtain  \eqref{diverr}.
\end{proof}

%============================
\begin{theorem}
 Let $(p, \uu,\s)$ solve problem \eqref{weakform} and $(p_h, \uu_h,\s_h)$ solve the semidiscrete mixed finite element approximation \eqref{semidiscreteform}. There is a positive constant $C$ such that for each $t\in (0,T)$,  
\beq\label{re4}
 \begin{split}
 &\norm {(\bar p -\bar p_h)(t)}_{H^{-1}(\Omega)}\le C\Big\{ h\norm{\theta(t)} + h \norm {(\pi\Psi -\Psi)(t)} \\
 &\quad + C\Upsilon^{\frac 12}\sqrt{\int_0^t \norm{\zeta(\tau)}_{L^\beta(\Omega)}d\tau}+ \int_0^t\norm{\nabla\cdot \left(\uu -\Pi\uu\right)(\tau) }d\tau \Big\}.
 \end{split}
 \eeq
 \end{theorem}
\begin{proof}
Let $\varphi\in H_0^1(\Omega)$ and $\pi \varphi \in W_h$
\beq\label{h1}
\begin{aligned}
(\bar p-\bar p_h,\varphi) &= (\bar p- \bar p_h,\varphi -\pi \varphi) + (\bar p- \bar p_h,\pi \varphi).
%\\&= -(\theta,\varphi -\pi \varphi)-(\vartheta,\varphi -\pi \varphi) + (\bar p-\bar p_h,\pi \varphi).
\end{aligned}
\eeq
Properties of projection allow us to bound  
\beq\label{h2}
\begin{aligned}
 (\bar p- \bar p_h,\varphi -\pi \varphi) \le C\norm{\bar p- \bar p_h} \norm{\varphi -\pi \varphi}
 %-(\theta,\varphi -\pi \varphi)-(\vartheta,\varphi -\pi \varphi) &\le C\norm{\theta} \norm{\varphi -\pi \varphi}\\
 &\le Ch\norm{\bar p- \bar p_h}\norm{\varphi}_{H^1}.   
\end{aligned}
\eeq
Since $\int_0^t \nabla\cdot \rho d\tau\in W_h $, the $L^2$-project shows that  
\beqs
\left(\int_0^t \nabla\cdot \rho d\tau , \pi \varphi\right) = \left(\int_0^t \nabla\cdot \rho d\tau , \varphi\right).
\eeqs
Using \eqref{Erreq1} with $w_h = \pi \varphi$ and definition of projections we find that
\begin{align*}
( \vartheta,\pi \varphi )&= \int_0^t ( \vartheta_t,\pi \varphi )d\tau =\int_0^t (\bar p_{h,t} -\bar p_t, \pi \varphi) d\tau\\
&= -\left(\int_0^t \nabla\cdot ( \uu_h -\uu) d\tau , \pi \varphi\right)  =-\left(\int_0^t \nabla\cdot \rho d\tau , \pi \varphi\right)\\
&=-\left(\int_0^t \nabla\cdot \rho d\tau ,\varphi\right).
\end{align*}
Thus
\beq\label{h3}
(\bar p-\bar p_h,\pi \varphi)=-(\vartheta,\pi \varphi) \le C\norm{\int_0^t \nabla \cdot\rho d\tau} \norm{\varphi}_{H^1}.   
\eeq
It follows from \eqref{h1}, \eqref{h2} and \eqref{h3} that
\begin{align*}
\frac{(\bar p-\bar p_h,\varphi)}{\norm {\varphi}_{H^1} } &\le Ch\norm{\bar p- \bar p_h}+ \norm{\int_0^t \nabla \cdot\rho d\tau}\\
&\le Ch\norm{\bar p- \bar p_h}+ \norm{\int_0^t \nabla \cdot(\uu_h-\uu) d\tau}+\norm{\int_0^t \nabla \cdot (\Pi \uu-\uu) d\tau}.
\end{align*}  
This,\eqref{diverr} and \eqref{mi2a} implies \eqref{re4}. 
\end{proof} 

%%=============================
Return to error estimate for vector gradient of pressure we have the following results  \begin{theorem}\label{suerr} Under the assumptions of Theorem \ref{mainres}. There exists a positive constant $C$ independent of $h$ such that  for each $0<t_0\le t\le T$,  
\beq\label{Bs2a}
\norm{(\s_h - \s)(t)}^2_{L^\beta(\Omega)}\le C  \Upsilon^{\gamma+ \frac 12}\Xi \sqrt{\int_0^t  \norm{ \zeta(\tau)}_{L^\beta(\Omega)}  d\tau} + \Upsilon^{\gamma+1} \norm {\zeta(t)}_{L^\beta(\Omega)}.
\eeq
Consequently, if $\s \in L^1\left(0,T; (W^{r+1,\beta}(\Omega))^d\right)$  then
\beq\label{Bs-sh2a}
\begin{split}
\norm{(\s_h - \s)(t)}_{L^\beta(\Omega)}&\le C\Upsilon^{\frac{2\gamma+1}{4} } \Xi^{\frac 1 2}h^{\frac{r+1}4}\left(\int_0^t  \norm{\s(\tau)}_{r+1,\beta}d\tau\right)^{\frac14} \\
    &\quad+ C\Upsilon^{\frac {\gamma+1}2 } h^{\frac{r+1}2}\sqrt{ \norm {\s(t)}_{r+1,\beta}} 
\end{split}
\eeq
and
 \beq\label{Bu-uh2a}
\begin{split}
\norm{(\uu_h - \uu)(t)}_{L^\beta(\Omega)}&\le C\Upsilon^{\frac{2\gamma+1}{4} }\Xi^{\frac 1 2}h^{\frac{r+1}4}\left(\int_0^t  \norm{\s(\tau)}_{r+1,\beta}d\tau\right)^{\frac14}\\
&+ C\Upsilon^{\frac {\gamma+1}2 } h^{\frac{r+1}2}\sqrt{\norm {\s(t)}_{r+1,\beta}} + Ch^{r+1} \norm{ \uu(t)}_{r+1,\beta}.
\end{split}
\eeq
\end{theorem}
%======================
\begin{proof}  
Thank to \eqref{Mn1}, \eqref{eqerr} and $L^2$-projection,  
\beqs
\begin{split}
\omega\norm{\s_h - \s}^2_{L^{\beta}(\Omega)} &\le \left( K(|\s_h| )\s_h - K(|\s|) \s ,  \s_h - \s \right)\\
%& =-(\vartheta_t, \vartheta)+\left( K(|\s_h| )\s_h - K(|\s|) \s , \zeta \right)\\
& =-(\bar p_{h,t}- \bar p_t, \vartheta)+\left( K(|\s_h| )\s_h - K(|\s|) \s , \zeta \right),
\end{split}
\eeqs
 from which, \eqref{mi1a}, \eqref{Bomega} and \eqref{keyb}. It follows that 
\beq
\begin{aligned}
\omega\norm{\s_h - \s}^2_{L^{\beta}(\Omega)}&\le  C (\norm{\bar p_{h,t}} +\norm{\bar p_t} )\norm{\vartheta} + C\Upsilon\norm {  \zeta}_{L^{\beta}(\Omega)}\\
&\le C \Xi\left(\Upsilon\int_0^t  \norm{ \zeta}_{L^{\beta}(\Omega)}  d\tau\right)^{\frac 12} + C\Upsilon \norm {\zeta}_{L^{\beta}(\Omega)}.
\end{aligned}
\eeq
Thus 
\beq\label{ssh2}
\norm{\s_h - \s}^2_{L^{\beta}(\Omega)}\le C \Upsilon^{\frac12}\Xi\omega^{-1} \left(\int_0^t  \norm{ \zeta}_{L^{\beta}(\Omega)}  d\tau\right)^{\frac 12} + C\Upsilon \omega^{-1}\norm {\zeta}_{L^{\beta}(\Omega)}.
\eeq
Note that from \eqref{omegadef} we have 
\beq\label{invOme}
\begin{split}
\omega^{-1} &\le C\left(1+ \norm{\s}_{L^{\beta}(\Omega)}+ \norm{\s_h}_{L^{\beta}(\Omega)}  \right)^a\\
&\le C\left(1+ \norm{\s_h}_{L^{\beta}(\Omega)}^\beta +  \norm{\s}_{L^{\beta}(\Omega)}^\beta  \right)^\gamma\le C\Upsilon^\gamma.
\end{split}
\eeq
 Substituting \eqref{invOme} into \eqref{ssh2}, we obtain \eqref{Bs2a}. 

Using \eqref{Bzeta} in \eqref{Bs2a} we obtain \eqref{Bs-sh2a}.  

To prove \eqref{Bu-uh2a} we use \eqref{Erreq2} with $\zz_h = \rho^{\beta-1}\in \tilde W_h:$    
\beqs
\norm{ \rho}_{L^{\beta}}^{\beta} =( \rho , \rho^{\beta-1} ) =( \uu_h-\uu , \rho^{\beta-1} ) = - \left( K(|\s_h| )\s_h - K(|\s|) \s ,  \rho^{\beta-1} \right).
\eeqs
Proposition~\ref{Lips} and H\"oder's inequality lead to  
\beqs
\norm{ \rho}_{L^{\beta }}^{\beta} \le C( |\s_h-\s|, \rho^{\beta-1})\le C\norm{\s-\s_h}_{L^{\beta}(\Omega)}\norm{ \rho}^{\beta-1}_{L^{\beta}(\Omega)}  
\eeqs
and hence 
\beq\label{Bro}
\norm{ \rho}_{L^{\beta}(\Omega)}\le C\norm{\s-\s_h}_{L^{\beta}(\Omega)}.
\eeq
Triangle inequality and  \eqref{Bro} yield 
\beqs
\norm{\uu_h -\uu}_{L^{\beta}(\Omega)} 
\le C( \norm{\rho}_{L^{\beta}(\Omega)} +\norm {\varrho}_{L^{\beta}(\Omega)} )\le C( \norm{\s-\s_h}_{L^{\beta}(\Omega)} +\norm {\varrho}_{L^{\beta}(\Omega)}). 
\eeqs
Therefore \eqref{Bu-uh2a} follows by using \eqref{Bs-sh2a} and \eqref{Bvarrho}. The proof is complete.    
\end{proof}    

{\bf Non-degenerate case.} In previous discussion, we developed error bounds based on the minimal regularity assumptions, using fairly weak norms on the error ($L^{\beta}(\Omega)(\Omega)$-norm). In following discussion we bounds errors in numerical solution in term of strong norms ($L^2$-norm), but make some assumption on the the regularity of solution. In particular, we assume that 
$$p, \Psi\in L^\infty(0,T; H^{r+1}(\Omega)) \text{ and }  \s\in L^\infty(0,T;(L^\infty(\Omega)\cap H^{r+1}(\Omega))^d). $$  
\begin{theorem}\label{NC}  Let $(p, \uu,\s)$ solve problem \eqref{weakform} and $(p_h, \uu_h,\s_h)$ solve the semidiscrete mixed finite element approximation \eqref{semidiscreteform}. Then there is a positive constant $C$ such that for each $t\in (0,T)$,  
 \beq\label{re1}
\begin{aligned}
&\norm {(p_h - p)(t)  }+ \sqrt{\int_0^t \norm{(\s_h -\s)(\tau)}^2  d\tau}\\
&\quad\quad\le Ch^{r+1}\left\{\norm {\Psi(t)}_{r+1} + \norm {\bar p(t) }_{r+1}+\sqrt{\int_0^t \norm{\s(\tau)}_{r+1}^2  d\tau}\right\}.
\end{aligned}
\eeq
and  \beq\label{re2}
\begin{aligned}
\sqrt{\int_0^t \norm{(\uu_h - \uu)(\tau)}^2 d\tau}
&\le Ch^{r+1}\Big\{\norm {\Psi(t)}_{r+1} + \norm {\bar p(t) }_{r+1}\\
&\quad+\sqrt{\int_0^t \norm{\uu(\tau)}_{r+1}^2 +\norm{\s(\tau)}_{r+1}^2  d\tau}\Big\}.
\end{aligned}
\eeq
\end{theorem}
%=====================================
\begin{proof} We use the equation \eqref{eqerr}. The regularity of solution enable us to bound term by term of equation \eqref{eqerr} as following  

 According to \eqref{Mono},     
\beqs
\left( K(|\s_h| )\s_h - K(|\s|) \s, \s_h -\s \right)\ge (1-a)\norm{ K^{\frac 12}(\max\{|\s|, |\s_h|\}) (\s_h -\s) }^2.
\eeqs
Using the fact that $K(\cdot)$ is bounded from below, we find that   
\beq\label{Mtn}
\left( K(|\s_h| )\s_h - K(|\s| \s), \s_h -\s \right)\ge k(1-a)\norm{ \s_h -\s }^2
\eeq
for some $k>0$. 

By H\"oder's inequality,\eqref{Lipchitz} and Young's inequality,       
\beq\label{mid}
\begin{split}
\left( K(|\s_h| )\s_h - K(|\s|) \s, \zeta \right)&\le C\norm{\s_h -\s}\norm{\zeta}\\
 &\le \frac{k(1-a)}2 \norm{ \s_h -\s}^2 + C \norm{\zeta}^2.
\end{split}
\eeq
Hence \eqref{Mtn}, \eqref{mid} and \eqref{eqerr} show that 
\beqs
\begin{split}
\frac d{dt}\norm {\vartheta}^2 + \norm{ \s_h -\s}^2 
\le C \norm {  \zeta}^2 .
\end{split}
\eeqs
Integrating this from $0$ to $t$, using $\vartheta(0)=0$, we obtain  
\beq\label{mid2}
\norm {\vartheta}^2 + \int_0^t  \norm{\s_h -\s }^2  d\tau \le C \int_0^t  \norm{ \zeta}^2 d\tau.
\eeq
Thus
\beq\label{mid2a}
\begin{split}
\norm {\bar p_h - \bar p  }^2 + \int_0^t  \norm{\s_h -\s }^2  d\tau \le \norm{ \theta}^2+C \int_0^t  \norm{ \zeta}^2 d\tau.
\end{split}
\eeq
Inequality \eqref{re1} follows from \eqref{trans1}, \eqref{Btheta}, \eqref{Bzeta} and \eqref{mid2a}.    
 
In \eqref{prj2}, select $\zz_h =\rho$ we obtain  
\beq\label{Brho}
\begin{aligned}
\norm{\rho}^2&=-\left( K(|\s_h| )\s_h - K(|\s|) \s, \rho\right)\\
&\le C(|\s_h -  \s|,|\rho|)\\
& \le C\norm{\s_h -  \s} \norm{\rho}.
\end{aligned}
\eeq
This leads to  
\beq\label{uu}
\begin{aligned}
\int_0^t\norm{\uu_h -\uu}^2 d\tau &\le  C \int_0^t  \norm{\varrho}^2+\norm{\rho}^2 d\tau\\&
 \le  C \int_0^t  \norm{\varrho}^2+\norm{\s_h -\s}^2 d\tau.
\end{aligned}
\eeq
We obtain \eqref{re2} by using \eqref{re1} and \eqref{Bvarrho} in \eqref{uu}. The proof is complete.  
 \end{proof}
\begin{theorem}\label{Bpinf} Let $(p, \uu,\s)$ solve problem \eqref{weakform} and $(p_h, \uu_h,\s_h)$ solve the semidiscrete problem \eqref{semidiscreteform}. If $p, \Psi\in L^\infty(0,T; W^{r+1, \infty}(\Omega))$ and $ \s\in L^2(0,T;H^{r+1}(\Omega))^d) $. Then there is a positive constant $C$ such that for each $t\in (0,T),$
\beqs
\begin{split}
\norm{(p-p_h)(t)}_{L^\infty(\Omega)}&\le Ch^{r+1}\Big( \norm{\bar p(t)}_{r+1, \infty}+\norm{\Psi(t)}_{r+1, \infty} \Big)\\
&\quad  + C h^r\sqrt{\int_0^t  \norm{\s(\tau)}^2_{r+1}d\tau}.
\end{split}
\eeqs
\end{theorem}
\begin{proof} It follows from \eqref{trans1}, \eqref{mid2}, \eqref{Btheta} and \eqref{Bzeta} that  
\begin{align*}
\norm{p-p_h}_{L^\infty(\Omega)}&\le \norm{\theta }_{L^\infty(\Omega)} +\norm{\vartheta}_{L^\infty(\Omega)}+\norm{\pi\Psi -\Psi}_{L^\infty(\Omega)} \\
 &\le Ch^{r+1}\Big( \norm{\bar p}_{r+1, \infty}+\norm{\Psi}_{r+1, \infty} \Big) + C h^{-1}\left(\int_0^t  \norm{\zeta}^2 d\tau\right)^{\frac 12}\\
&\le Ch^{r+1}\Big( \norm{\bar p}_{r+1, \infty}+\norm{\Psi}_{r+1, \infty} \Big) + C h^r\left(\int_0^t  \norm{\s}^2 d\tau\right)^{\frac 12}
\end{align*}
which completes the proof. 
\end{proof}
\begin{theorem}\label{Bsu}  Assume $(\bar p^0,\uu^0,\s^0 )\in W\times V\times \tilde W$ and $(\bar p_h^0,\uu^0_h,\s_h^0 )\in W_h\times V_h\times \tilde W_h$. Then there are positive constants $C$ independent of $h$ such that for each $0<t_0\le t\le T$,  
\beq\label{sErr1}
\norm{(\s_h - \s)(t)}\le C\Xi^{\frac12}h^{\frac {r+1}2}\left(\int_0^t  \norm{\s(\tau)}^2_{r+1} d\tau\right)^{\frac14} + Ch^{r+1}\norm {\s(t)}_{r+1}
\eeq
and 
\beq\label{uErr1}
\begin{split}
\norm{(\uu_h - \uu)(t)}&\le C\Xi^{\frac12}h^{\frac {r+1}2}\left(\int_0^t  \norm{\s(\tau)}^2_{r+1} d\tau\right)^{\frac14} \\
&\quad+ Ch^{r+1}\left(\norm {\s(t)}_{r+1} +\norm {\uu(t)}_{r+1}\right). 
\end{split}
\eeq
\end{theorem}
%-------------------------------------------
\begin{proof}
Thank to \eqref{Mtn}, \eqref{eqerr} and $L^2$-projection,  
\beqs
\begin{split}
C\norm{\s_h - \s}^2 &\le \left( K(|\s_h| )\s_h - K(|\s|) \s ,  \s_h - \s \right)\\
& =-(\vartheta_t, \vartheta)+\left( K(|\s_h| )\s_h - K(|\s|) \s , \zeta \right)\\
& =-(\bar p_{h,t}- \bar p_t, \vartheta)+\left( K(|\s_h| )\s_h - K(|\s|) \s , \zeta \right)\\
&\le C (\norm{\bar p_{h,t}} +\norm{\bar p_t} )\norm{\vartheta} + C\norm {  \zeta}^2.
\end{split}
\eeqs
% Using \eqref{mid} yields
%\beq
%\norm{\s_h - \s}^2\le  C (\norm{\bar p_{h,t}} +\norm{\bar p_t} )\norm{\vartheta} + C\norm {  \zeta}^2.
%\eeq    
According to \eqref{Bpht} and \eqref{Bpt}, 
$\norm{\bar p_{h,t}} +\norm{\bar p_t}\le C\Xi$. This and \eqref{mid2} give     
\beq\label{s-sh}
\norm{\s_h - \s}^2\le C\Xi\left(\int_0^t  \norm{ \zeta}^2  d\tau\right)^{\frac 12} + C\norm {\zeta}^2.
\eeq
 Hence \eqref{sErr1} holds.
 
 We have 
 $$
 \norm{\uu_h - \uu}\le \norm{\rho}+\norm{\varrho}.  
 $$
  Using \eqref{Brho} and \eqref{s-sh}, 
 \beqs
 \norm{\rho} \le C \norm{ \s_h - \s}\le C\Xi^{\frac 12}\left(\int_0^t  \norm{ \zeta}^2  d\tau\right)^{\frac 14} + C\norm {\zeta}.
 \eeqs
 This leads to      
 \beq\label{uuhdiff}
 \norm{\uu_h - \uu}\le C\Xi^{\frac 12}\left(\int_0^t  \norm{ \zeta}^2  d\tau\right)^{\frac 14} + C\norm {\zeta}+ \norm {\varrho}.
 \eeq
 Combining \eqref{uuhdiff}, \eqref{Bzeta} and \eqref{Bvarrho}, we obtain \eqref{uErr1}.
 \end{proof}     
 %======================= 
\subsection{Error analysis for fully discrete scheme} 
In this subsection, we present some convergence results to the fully discrete scheme for the degenerate case and super convergence for the nondegenerate case.  

Let
$\bar p^n(\cdot) = \bar p(\cdot,t_n)$, $\vv^n(\cdot) = \vv(\cdot,t_n)$ and $\uu^n(\cdot) = \uu(\cdot,t_n)$ be the
true solution evaluated at the discrete time levels.  We will also
denote $\pi p^n \in W_h$, $\pi \s^n \in \tilde W_h$ and $\Pi \uu^n \in V_h$ to be the projections
of the true solutions at the discrete time levels.  
%===========================

We rewrite \eqref{weakform} with $t=t_n$. Using the definitions of projections and assumption that $\nabla\cdot V_h \subset W_h $, standard
manipulations show that the true solution satisfies the discrete equation
\begin{subequations}\label{fulprjsys}
\begin{align}
\label{fulprj1}& \left( \frac{\pi\bar p^n - \pi\bar p^{n-1}}{\Delta t }, w_h\right)  +  \left(\nabla\cdot \Pi\uu^n, w_h\right) =(f^n-\Psi_t^n, w_h) +(\epsilon^n,w_h),\\
\label{fulprj2}& (\Pi\uu^n, \zz_h)  + ( K(|\s^n| )\s^n ,\zz_h)=0, \\
 \label{fulprj3} &(\pi\s^n,\vv_h)  + ( \pi \bar p^n , \nabla\cdot  \vv_h )=(\nabla \Psi^n, \vv_h ),
\end{align}
\end{subequations}  
where $\epsilon^n$ is the time truncation error of order $\Delta t$. 
%==================================
\begin{theorem}\label{fulErr} Assume $(\bar p^0,\uu^0,\s^0 )\in W\times V\times \tilde W$ and $(\bar p_h^0,\uu^0_h,\s_h^0 )\in W_h\times V_h\times \tilde W_h$. Let $(p, \uu,\s)$ solve problem \eqref{weakform} and $(p_h^n, \uu_h^n,\s_h^n)$ solve the fully  discrete mixed finite element approximation \eqref{fullydiscreteform} for each time step $n$, $n=1\ldots, N$.  
There exists a positive constant $C$ independent of $h$ and $\Delta t$  such that if  the $\Delta t$ is sufficiently small then   
\beq\label{fulerr1}
\begin{split}
\norm{ \bar p^m_h- \bar p^m}\le C\left(\Upsilon\sum_{n=1}^m\Delta t \norm{\zeta^n}^2_{L^{\beta}(\Omega)}
\right)^{\frac 12}
+ C\norm{\theta^m}+C\Delta t
\end{split}
\eeq
for all $m=1,\dots, N.$

Consequently,  if $p^n,\Psi^n\in H^{r+1}(\Omega)$  and  $\s^n\in \left(W^{r+1,\beta}(\Omega)\right)^d$ for $n=1,\dots, N$ then for $m$ between $1$ and $N$, 
\beq\label{fulerr2}
\norm{  p_h^m -  p^m} \le C(h^{r+1}+\Delta t). 
\eeq
\end{theorem}
\begin{proof} 
Subtract \eqref{fulprjsys} from \eqref{fullydiscreteform}, in the resultants using $w_h= \vartheta^n, \zz_h = \eta^n,\vv_h= \rho^n$  we obtain the error equations: 
\begin{subequations}
\begin{align}
&\label{b1} \left( \frac{\vartheta^n- \vartheta^{n-1}}{\Delta t}, \vartheta^n \right) +  \left(\nabla\cdot \rho^n, \vartheta^n\right) =(\epsilon^n,\vartheta^n), \\
&\label{b2} (\rho^n, \eta^n )  + \left( K(|\s_h^n| )\s_h^n -K(|\s^n|) \s^n ,  \eta^n \right)=0,\\ 
&\label{b3} (\eta^n,\rho^n)  + ( \vartheta^n , \nabla\cdot \rho^n )= 0.
\end{align}
\end{subequations} 
Combining \eqref{b1}--\eqref{b3} gives 
\beqs
\norm{ \vartheta^n}^2 + \Delta t \left( K(|\s_h^n| )\s_h^n -K(|\s^n|) \s^n ,  \eta^n \right) = ( \vartheta^n, \vartheta^{n-1}  )+\Delta t(\epsilon^n,\vartheta^n).
\eeqs
We rewrite this equation as form
\beq\label{c11}
\begin{aligned}
&\norm{ \vartheta^n}^2 + \Delta t \left( K(|\s_h^n| )\s_h^n -K(|\s^n|) \s^n ,  \s_h^n-\s^n \right) \\
&\quad\quad= (\vartheta^n, \vartheta^{n-1}  )+ \Delta t\Big\{ \left( K(|\s_h^n| )\s_h^n -K(|\s^n|) \s^n ,  \zeta^n \right)+(\epsilon^n,\vartheta^n)\Big\}.
\end{aligned}
\eeq
The second term of \eqref{c11}, using \eqref{Mono}, gives 
\beq\label{Wssh}
\left( K(|\s_h^n| )\s_h^n -K(|\s^n|) \s^n ,  \s_h^n-\s^n \right) \ge C\omega^n  \norm{\s_h^n -\s^n}^2_{L^{\beta}(\Omega)}.
\eeq
where $\omega^n =\omega (t_n)$ defined as in \eqref{omegadef}. 

The right hand side of \eqref{c11}, using Young's inequality, \eqref{mi1a}, \eqref{keya} and \eqref{Bomega}, gives   
\begin{multline} \label{RHSc11}
(\vartheta^n, \vartheta^{n-1}  )+ \Delta t\left( \left( K(|\s_h^n| )\s_h^n -K(|\s^n|) \s^n ,  \zeta^n \right)+(\epsilon^n,\vartheta^n)\right)\\
\le \frac 12 \left(\norm{\vartheta^n}^2 +\norm{\vartheta^{n-1}}^2   \right)+
\Delta t\Big\{C\Upsilon\norm{\zeta^n}^2_{L^{\beta}(\Omega)} +\frac12\left(\norm{\vartheta^n}^2 +\norm{\epsilon^n}^2\right) \Big\}.
\end{multline}
From \eqref{c11}, \eqref{Wssh} and \eqref{RHSc11}, we obtain  
\begin{align*}
\norm{ \vartheta^n}^2 - \norm{\vartheta^{n-1}}^2&+ C\Delta t\omega^n \norm{\s_h^n-\s^n}^2_{L^{\beta}(\Omega)} \\
&\le\Delta t \norm{ \vartheta^n}^2 + C \Delta t \Big(\Upsilon \norm{\zeta^n}^2_{L^{\beta}(\Omega)}+\norm{\epsilon^n}^2\Big). 
\end{align*}
Summing  over $n$ gives
\begin{align*}
(1-\Delta t)\norm{ \vartheta^m}^2 &+ C \sum_{n=1}^m \Delta t  \omega^n\norm{\s_h^n-\s^n}^2_{L^{\beta}(\Omega)}\\
 &\le\sum_{n=1}^{m-1} \Delta t \norm{ \vartheta^n}^2  
+ C \sum_{n=1}^m \Delta t \Big(\Upsilon  \norm{\zeta^n}^2_{L^{\beta}(\Omega)}+\norm{\epsilon^n}^2\Big). 
\end{align*}
By discrete Gronwall's inequality,  
 \beqs
\norm{\vartheta^m}^2 + C\sum_{n=1}^m\Delta t \omega^n\norm{\s_h^n-\s^n}^2_{L^{\beta}(\Omega)}
\le C\sum_{n=1}^m\Delta t \Big(\Upsilon \norm{\zeta^n}^2_{L^{\beta}(\Omega)}+\norm{\epsilon^n}^2\Big). 
\eeqs
Therefore
 \begin{align*}
 \norm{ \bar p^m_h- \bar p^m}^2 &+  \sum_{n=1}^m\Delta t \omega^n\norm{\s_h^n-\s^n}^2_{L^{\beta}(\Omega)}\\
  &\le C\Upsilon\sum_{n=1}^m\Delta t  
\norm{\zeta^n}^2_{L^{\beta}(\Omega)}
+ \norm{\theta^m}^2+C(\Delta t)^2. 
\end{align*}
This implies \eqref{fulerr1}. 

From \eqref{fulerr1} and triangle inequality we find that 
\beq\label{as}
\norm{ p^m_h-  p^m}\le C\left(\Upsilon\sum_{n=1}^m\Delta t \norm{\zeta^n}^2_{L^{\beta}(\Omega)}
\right)^{\frac 12}
+ C\norm{\theta^m}+C\Delta t +\norm{\pi\Psi^m-\Psi^m}.
\eeq
The project properties and \eqref{as} imply \eqref{fulerr2}.
\end{proof}
%=====================
\begin{theorem}\label{Derr}
Under assumptions of Theorem \ref{fulErr}. If $\s^n\in (W^{r+1,2}(\Omega))^d$ for $n=1,\dots, N$ then there is positive constant $C$ independent of $h$ and time step such that if $\Delta t$ sufficiently small then for   $m$ between $1$ and $N$,  
\beq\label{Dsuh}
\norm{\s_h^m - \s^m}_{L^{\beta}(\Omega)}+\norm{\uu_h^m - \uu^m}_{L^{\beta}(\Omega)}\le C( h^{\frac{r+1}2}+ \sqrt{\Delta t} ) .
\eeq
\end{theorem}
%======================
\begin{proof}  
Recall that the true solution satisfies the discrete equations 
\begin{subequations}\label{tem1}
\begin{align}
\label{te1}& \left( p_t^n, w_h\right)  +  \left(\nabla\cdot \Pi\uu^n, w_h\right) =(f^n, w_h)  , &&\forall w_h\in W_h\\
\label{te2}& (\Pi\uu^n, \zz_h)  + ( K(|\s^n| )\s^n ,\zz_h)=0, &&\forall \zz_h\in \tilde W_h,\\
 \label{te3} &(\pi\s^n,\vv_h)  + ( \pi p^n , \nabla\cdot  \vv_h )=0, &&\forall \vv_h\in V_h,
\end{align}
\end{subequations} 
Subtracting \eqref{fullydiscreteform} from \eqref{tem1}, 
choosing $w_h=\vartheta^n$, $\zz_h =\eta^n$, $\vv_h =\rho^n$, we obtain  
\begin{subequations}\label{tem4}
\begin{align}
&\label{te41}    \left( \frac{ p_h^n -  p_h^{n-1}}{\Delta t } -p_t^n, \vartheta^n\right) + \left(\nabla \cdot \rho^n, \vartheta^n\right) =0, \\
&\label{te42}     (\rho^n, \eta^n)  + ( K(|\s_h^n| )\s_h^n - K(|\s^n| )\s^n ,\eta^n )=0,\\
 &\label{te43}     (\eta^n,\rho^n)  + ( \vartheta^n , \nabla\cdot \rho^n )=0.
\end{align}
\end{subequations} 
Above equations yield   
\beq\label{erq}
\left( \frac{ p_h^n -  p_h^{n-1}}{\Delta t } -p_t^n, \vartheta^n\right) + \left( K(|\s_h^n| )\s_h^n - K(|\s^n| )\s^n ,\eta^n \right)=0.
\eeq
We use \eqref{Mn1}, \eqref{erq} to find that     
\beq\label{pior-est}
\begin{split}
\omega^n\norm{\s^n_h - \s^n}^2_{L^{\beta}(\Omega)} &\le \left( K(|\s^n_h| )\s^n_h - K(|\s^n|) \s^n ,  \s^n_h - \s^n \right)\\
&=\left( K(|\s^n_h| )\s^n_h - K(|\s^n|) \s^n ,  \eta^n \right)+ \left( K(|\s^n_h| )\s^n_h - K(|\s^n|) \s^n ,   \zeta^n \right)\\
& =-\left(\frac{p_{h}^n - p_h^{n-1}}{\Delta t} - p_t^n, \vartheta^n\right)+\left( K(|\s^n_h| )\s^n_h - K(|\s^n|) \s^n , \zeta^n \right).
\end{split}
\eeq
Due to \eqref{mi1a}, Cauchy-Schwartz and triangle inequality, one has  
\beqs
\omega^n\norm{\s^n_h - \s^n}^2_{L^{\beta}(\Omega)}\le C \left((\Delta t)^{-1} \norm{ p_{h}^n - p_h^{n-1}} +\norm{p_t^n} \right)\norm{\vartheta^n} + \Upsilon \norm {  \zeta^n}_{L^{\beta}(\Omega)}. 
\eeqs
Since 
\begin{align*}
(\Delta t)^{-1}\norm{\bar p_{h}^m - p_h^{m-1}} &=(\Delta t)^{-1}\norm{ \int_{t_{m-1}}^{t_m} p_{h,t} dt}\\
& \le (\Delta t)^{-1} \int_{t_{m-1}}^{t_m} \norm{\bar p_{h,t}} dt\\
&\le \sup_{[T/N,T]}\norm{\bar p_{h,t}|}\le \Xi, 
\end{align*}
and also note that $ \norm{\bar p_t^m} \le \sup_{[T/N,T]}\norm{\bar p_t}\le \Xi $, we have 
\beqs
\begin{split}
\omega^n\norm{\s_h^m - \s^m}^2_{L^{\beta}(\Omega)}
&\le C \Xi \norm{\vartheta^m} + \Upsilon \norm {  \zeta^m}^2_{L^{\beta}(\Omega)}\\
&\le C \Xi \left[\sum_{n=1}^m\Delta t \left(\Upsilon \norm{\zeta^n}^2_{L^{\beta}(\Omega)}+\norm{\epsilon^n}^2\right)\right]^{\frac 12} + \Upsilon \norm {\zeta^m}^2_{L^{\beta}(\Omega)}.
\end{split}
\eeqs
Using \eqref{invOme}, we obtain 
\beqs%\label{fDsh}
\norm{\s_h^m - \s^m}^2_{L^{\beta}(\Omega)}\le C \Upsilon^\gamma \Xi \left[ \left(\Upsilon\sum_{n=1}^m\Delta t  \norm{\zeta^n}^2_{L^{\beta}(\Omega)}\right)^{\frac 12}+\Delta t    \right] + C \Upsilon^{\gamma+1} \norm {\zeta^m}^2_{L^{\beta}(\Omega)}.
\eeqs
It follows that %from \eqref{fDsh} and \eqref{Bzeta} that
\beq\label{Dsh}
\begin{split}
\norm{\s_h^m - \s^m}_{L^{\beta}(\Omega)}&\le C\Upsilon^{\frac{2\gamma+1}{4} }\Xi^{\frac 1 2}h^{\frac{r+1}2}\left( \sum_{n=1}^m  \Delta t\norm{\s^n}^2_{r+1,\beta}d\tau\right)^{\frac12}\\
& \quad+ C\Upsilon^{\frac {\gamma+1}2 } h^{r+1}\norm {\s^m}_{r+1,\beta}
+ C \Xi\Upsilon^\gamma(\Delta t)^{\frac 12}\\
&\le C(h^{\frac{r+1}2 }+\sqrt{\Delta t} ).
\end{split}
\eeq

Now we subtract  \eqref{fulprj2} from \eqref{fully2}, use $\zz_h =(\rho^m) ^{\beta-1} $, we have   
 \beqs
 \left(\rho^m, (\rho^m)^{\beta-1}\right) + \left( K(|\s_h^m| )\s_h^m -K(|\s^m|) \s^m ,  (\rho^m)^{\beta-1} \right)=0.
 \eeqs 
In a similar way as in Theorem \ref{suerr} we find that   
 \beqs
 \norm{ \rho^m}_{L^{\beta}(\Omega)} \le C\norm{\s^m -\s_h^m}_{L^{\beta}(\Omega)}.
 \eeqs
 which implies
 \beq\label{Duh}
\norm{\uu_h^m -\uu^m}_{L^{\beta}(\Omega)}\le  C( \norm{\s^m-\s_h^m}_{L^{\beta}(\Omega)} +\norm {\varrho^m}_{L^{\beta}(\Omega)})\le C(h^{\frac{r+1}2 }+\sqrt{\Delta t} ). 
\eeq
 Thus \eqref{Dsh} and \eqref{Duh} lead to \eqref{Dsuh}. The proof is complete. 
\end{proof}
Finally we derives $L^2$-estimates for  $p_h^m - p^m,$ $\s_h^m - \s^m$ and $\uu_h^m - \uu^m $ in the nondegenerate case.  As results of Theorem~\ref{NC} and   \ref{Bsu} we can obtain the following error estimates   
\begin{theorem}\label{fullyErr} Suppose $(\bar p^0,\uu^0,\s^0 )\in W\times V\times \tilde W$ and $(\bar p_h^0,\uu^0_h,\s_h^0 )\in W_h\times V_h\times \tilde W_h$. Let $(p, \uu,\s)$ solve problem \eqref{weakform} and $(p_h^n, \uu_h^n,\s_h^n)$ solve the fully  discrete mixed finite element approximation \eqref{fullydiscreteform} for each time step $n$.  Assume $p^n, \Psi^n \in H^{r+1}(\Omega)$ and  $\s^n\in (L^\infty(\Omega)\cap H^{r+1}(\Omega))^d$, for $n=1,\dots, N$. There exists a positive constant $C$ independent of $h$ such that if $\Delta t$ sufficiently small then for   $m$ between $1$ and $N$,  

{\rm (i)}
\beq\label{fulre2}
\begin{aligned}
\norm{  p_h^m - p^m} &+ \left(\sum_{n=1}^m \Delta t \norm{\s_h^n -\s^n}^2\right)^{\frac 12}\\
& +\left(\sum_{n=1}^m\Delta t \norm{\uu_h^n - \uu^n}^2\right)^{\frac12}\le C(h^{r+1}+\Delta t). 
\end{aligned}
\eeq
 
 {\rm (ii)}
\beq
\norm{\s_h^m - \s^m}+\norm{\uu_h^m - \uu^m}\le C( h^{\frac{r+1}2}+ \sqrt{\Delta t} ) .
\eeq
\end{theorem}
%==============================
\section{Numerical results}\label{num-res}
In this section, we give a simple numerical result illustrating the convergence
theory. We test the convergence of our method with the Forchheimer two term law. 
For simplicity, consider $g(s)=1+s, s\ge 0$. Equation \eqref{eq4} shows that
 $ s= \frac {-1 +\sqrt{1+4\xi}}{2}$ and  
$$
K(\xi) =\frac1{g(s(\xi)) } =\frac {2}{1+\sqrt{1+4\xi}}.
$$ 

Since we analyze a first order time discretization, we consider the lowest order mixed method. Here we use the lowest order Raviart-Thomas mixed finite element on the unit square in two dimensions. Let $x=(x_1,x_2), \Omega=[0,1]^2$. The  analytical solution is chosen by 
\begin{align*}
p(x,t)&=tx_1(1-x_1)x_2(1-x_2), \\
\s(x,t)&=t ((1-2x_1)x_2(1-x_2), x_1(1-x_1)(1-2x_2)    ),\\
 \uu(x,t)&= \frac{2\s(x,t)}{1+\sqrt{1 + 4|\s(x,t)|}} 
\end{align*}
for all $x\in\Omega,t\in [0,1]$. The forcing term $f$ is determined accordingly to the analytical solution by 
$
p_t - \nabla \cdot\uu = f, (x,t)\in \Omega\times[0,1]. 
$ 
The initial data $p(x,0)=0$ and boundary data $p(x,t)=0$ for all $(x,t)\in \partial \Omega \times[0,1].$ 

We divided the unit square into an $ N\times N$ mesh of squares, each then subdivided into two right triangles. For each mesh, we solved the generalized Forchheimer equation numerically. The error control in each nonlinear solve is $\varep =10^{-6}$.  Our problem is solved at each time level start at $t=0$ until final time $t = 1$. At this time, we measured the $L^2$-errors of pressure and $L^{\beta}$-errors of gradient of pressure and flux with $\beta = 2 - a= 2- \frac{{\rm deg}( g)}{{\rm deg} (g) + 1}= \frac 3 2$.   The numerical results are listed as the following table.   
  
  \vspace{0.2cm}
\begin{center}
\begin{tabular}{c|c| c|c| c|c|c}
%\hline
N    & $\norm{p-p_h}$ & Rates & $\norm{\s-\s_h}_{L^{\beta}(\Omega)}$ &  Rates & $\norm{\uu-\uu_h}_{L^{\beta}(\Omega)}$ &  Rates \\
\hline
4& 0.00070985	&-&0.052427&-&0.0492867&-\\
8&  0.000307278&	2.31&	0.0277362&	1.89&	0.0267065	&1.85 \\
16& 0.000142125&	2.16&0.016163&	1.72	 &0.0158135&	1.69 \\
32& 7.44E-05&	1.91&	0.0115456&	1.4&	0.0113782&	1.39\\
64&3.83E-05&	1.94&	0.01003&	1.15&	0.00990809&	1.15\\
128&1.93E-05&	1.99&	0.00959868&	1.04&	0.00948783&	1.04\\
256&9.65E-06&	2.00&	0.00948485&	1.01&	0.00937671&	1.01\\
512& 4.83E-06&	2.00&	0.0094558&	1.00&	0.00934833&	1.00  \\
%\hline
\end{tabular}

\vspace{0.2cm}
Table 1. Convergence study for generalized Forchheimer equation in 2D.

\end{center}

\myclearpage
\appendix\section*{}
We state the parabolic embedding and fast decaying geometry sequences lemmas which are used in our proof of Theorem~\ref{pinf}. Let us denote throughout $Q_T=\Omega\times(0,T)$.
 \begin{lemma}[cf. \cite{HK}]\label{Sob4} 
Assume
$
\mu \ge 2 \quad \text{and}\quad  \mu>\frac{ad}{\beta}.
$
Let 
\beqs\label{expndef}
q=\mu\left(1+\frac{\beta}d\right)-a.
\eeqs
Then
\beqs\label{nonzerobdn}
\|u\|_{L^q(Q_T)}\le C(1+\delta T)^{1/q}[[u]],
\eeqs
where $\delta=1$ in general, $\delta=0$ in case $u$ vanishes on the boundary $\partial \Omega$, and
\beqs\label{udouble}
[[u]]=\max_{[0,T]}\|u(t)\|_{L^\mu(\Omega)}+\left(\int_0^T\int_\Omega |u|^{\mu-2}|\nabla u|^{\beta}dx dt\right)^\frac 1{\mu-a}.
\eeqs
\end{lemma}
%=================================================================%

\begin{lemma}[cf. \cite{HKP1}] \label{multiseq} 
Let $\{Y_i\}_{i=0}^\infty$ be a sequence of non-negative numbers satisfying
\beqs \label{mABi}
Y_{i+1}\le \sum_{k=1}^m A_k B_k^i  Y_i^{1+\mu_k}, \quad 
i =0,1,2,\cdots,
\eeqs
where  $A_k>0$, $B_k>1$ and $\mu_k>0$ for $k=1,2,\ldots,m$.
Let $B=\max\{B_k : 1\le k\le m\}$ and $\mu=\min\{\mu_k : 1\le k\le m\}$. Then the following statements hold true.
\beqs\label{Y0andD}
\text{If}\quad \sum_{k=1}^m  A_k Y_0^{\mu_k} \le B^{-1/\mu}
\quad\text{then } \lim_{i\to\infty} Y_i=0.
\eeqs
In particular,
\begin{align*}\label{Y0mcond}
\text{if}\quad  Y_0\le \min\{ (m^{-1} A_k^{-1} B^{-\frac 1 {\mu}})^{1/\mu_k} : 1\le k\le  m\}
\quad \text{then } \lim_{i\to\infty} Y_i=0.
\end{align*} 
\end{lemma}
%==========================
\begin{proof}
\noindent{\bf Proof of Theorem \ref{pinf}.}
We follow De Giorgi's technique (see \cite{LadyParaBook68}). 
First, we rewrite \eqref{maineq} as
\beq\label{eqgamma}
\bar p_t - \nabla\cdot (K(|\nabla p|)\nabla p) = f - \Psi_t.
\eeq

For any $k\ge 0$, let 
\beqs
\bar p^{(k)}=\max\{\bar p-k,0\}, \quad S_{k}(t)=\{ x\in \Omega: \bar p^{(k)}(x,t)\ge 0\}, \quad \sigma_k =\int_0^T |S_k(t)|dt.\eeqs
Let $k\ge \|\bar  p_0\|_{L^\infty}$. Then $\bar p^{(k)}(x,0)=0.$
Multiplying \eqref{eqgamma} by $|\bar p^{(k)}|^{\mu-1}$ and integrating over the domain $\Omega$ give
\beq\label{base0}
\begin{aligned}
\frac 1\mu \ddt \int_\Omega |\bar p^{(k)}|^\mu dx + (\mu-1)\int_\Omega |\bar p^{(k)}|^{\mu-2} K(|\nabla p|)\nabla p\cdot \nabla \bar p^{(k)} dx\\
= \int_\Omega (f-\Psi_t) |\bar{p}^{(k)}|^{\mu-1} dx.
\end{aligned}
\eeq
Since $\nabla \bar p^{(k)} = \nabla \bar p = \nabla p -\nabla \Psi$, \eqref{base0} implies  
\beq\label{est1}
\begin{aligned}
& \ddt \int_\Omega |\bar p^{(k)}|^\mu dx + \int_\Omega |\bar p^{(k)}|^{\mu-2} K(|\nabla p|)|\nabla p|^2 dx\\
&\quad\le C\int_\Omega |f-\Psi_t| |\bar{p}^{(k)}|^{\mu-1} dx +C\int_\Omega |\bar p^{(k)}|^{\mu-2} K(|\nabla p|)|\nabla p| |\nabla \Psi| dx.
\end{aligned}
\eeq
Using \eqref{i:ineq2}, we have
\begin{multline}\label{est2}
|\bar p^{(k)}|^{\mu-2} K(|\nabla p|)|\nabla p|^2
\ge C\left(|\bar p^{(k)}|^{\mu-2} |\nabla p|^{\beta}  -|\bar p^{(k)}|^{\mu-2}\right)\\
\ge C\left( |\bar p^{(k)}|^{\mu-2} |\nabla\bar p^{(k)}|^{\beta}  - |\bar p^{(k)}|^{\mu-2} |\nabla \Psi|^{\beta} -|\bar p^{(k)}|^{\mu-2}\right).
\end{multline}
Also    
\begin{align*}
 |\bar p^{(k)}|^{\mu-2} K(|\nabla p|)|\nabla p| |\nabla \Psi|&\le  |\bar p^{(k)}|^{\mu-2} (|\nabla\bar p|^{\beta-1}+  |\nabla \Psi|^{\beta-1}) |\nabla \Psi| \\
&\le  |\bar p^{(k)}|^{\mu-2} |\nabla\bar p|^{\beta-1}  |\nabla \Psi| + |\bar p^{(k)}|^{\mu-2}  |\nabla \Psi|^{\beta} .
\end{align*}
Young's inequality provides 
$$
|\nabla\bar p|^{\beta-1}|\nabla \Psi|\le \varep |\nabla\bar p^{(k)}|^{\beta}+ C\varep^ {1-\beta}|\nabla\Psi|^{\beta}.
$$ 
Thus   
\beq \label{est3}
\begin{aligned}
& |\bar p^{(k)}|^{\mu-2} K(|\nabla p|)|\nabla p| |\nabla \Psi| \\
&\quad\le \varep |\bar p^{(k)}|^{\mu-2} |\nabla\bar p^{(k)}|^{\beta} +C(1+\varep^ {1-\beta})|\bar p^{(k)}|^{\mu-2}  |\nabla \Psi|^{\beta}.
\end{aligned}
\eeq
Combining \eqref{est1}, \eqref{est2} and \eqref{est3}, selecting $\varep =C/2$ we obtain
\begin{align*}
& \ddt \int_\Omega |\bar p^{(k)}|^\mu dx + \int_\Omega |\bar p^{(k)}|^{\mu-2} |\nabla \bar p^{(k)}|^{\beta}dx\\
&\quad\le C\int_\Omega |f-\Psi_t| |\bar{p}^{(k)}|^{\mu-1} dx + C \int_\Omega |\bar p^{(k)}|^{\mu-2} ( |\nabla \Psi|^{\beta} + 1) dx\\
&\quad\le \delta  \int_\Omega |\bar p^{(k)}|^\mu dx  
 + C \delta^{1-\mu} \int_\Omega |f-\Psi_t|^\mu \chi_k dx \\
 &\qquad\qquad\qquad\qquad+  C \delta^{1-\frac \mu 2} (1+ \|\nabla \Psi\|_{L^\infty})^\frac{\beta\mu}{2}  |S_k(t)|\\
&\quad\le \delta  \int_\Omega |\bar p^{(k)}|^\mu dx  
 + C \delta^{1-\mu} \|f-\Psi_t\|_{L^{\mu+1}}^\mu |S_k(t)|\\
 &\qquad\qquad\qquad\qquad +  C \delta^{1-\frac \mu 2} (1+ \|\nabla \Psi\|_{L^\infty})^\frac{\beta\mu}{2}  |S_k(t)|.
\end{align*}
Here $\chi_k(t)$ is the characteristics function of $S_k(t)$. 
In previous inequality integrating from $0$ to $T$ and selecting $\delta=1/(2T)$ we find that
\begin{align*}
&\sup_{[0,T]} \int_\Omega |\bar p^{(k)}|^\mu dx + \int_0^T\int_\Omega |\bar p^{(k)}|^{\mu-2}|\nabla \bar p^{(k)}|^{\beta} dxdt\\
&\quad\le C\int_0^T \left( T^{\mu-1} \|f-\Psi_t\|_{L^{\mu+1}}^\mu  + T^{\frac \mu 2-1}(1+ \|\nabla \Psi\|_{L^\infty}^\frac{\beta\mu}{2})\right)|S_k(t)| dt.
\end{align*}
Let
\begin{align}
F_k &=\sup_{[0,T]} \int_\Omega |\bar p^{(k)}|^\mu dx + \int_0^T\int_\Omega |\bar p^{(k)}|^{\mu-2}|\nabla \bar p^{(k)}|^{\beta} dxdt,\quad \sigma_k=\int_0^T |S_k(t)|dt,\nonumber\\
\label{Edef} \mathcal E_T & = T^{\mu-1}\|f-\Psi_t\|_{L_t^\infty (0,T;L_x^{\mu+1})}^\mu 
+T^{\frac \mu 2-1}(1+ \|\nabla \Psi\|_{L_t^\infty(0,T; L_x^\infty)}^\frac{\beta\mu}{2}).
\end{align}
Then $ F_k \le C \mathcal E_T \sigma_k. $
Let $\mu>0$ be sufficient large, $q$  as in \eqref{expndef}. By Lemma~\ref{Sob4}:
\beq\label{general}
\| \bar p^{(k)}\|_{L^q(Q_T)}\le C(1+ T)^{1/q} (F_k^{1/\mu}+F_k^{1/(\mu-a)}).
\eeq
Let $k_i=M_0(2-2^{-i})$, for $i=0,1,2\ldots$, then $k_i$ is increasing in $i$, $S_{k_i}$ and $\sigma_{k_i}$ are decreasing.

Note that we want $k_0=M_0\ge \|\bar p_0\|_{L^\infty}$. By definition, 
\beq\label{kk}
\|\bar  p^{(k_i)}\|_{L^q(Q_T)} \ge \|\bar  p^{(k_i)}\|_{L^q(\mathcal Q_{k_{i+1}})} \ge  (k_{i+1}-k_{i})\sigma_{k_{i+1}}^{1/q},
\eeq
where $\mathcal Q_k=\{ (x,t)\in U\times(0,T): p(x,t)>k\}$.

Combining \eqref{kk} and \eqref{general}, we obtain 
\beqs
\sigma_{k_{i+1}}^{1/q}\le \frac{C(1+T)^{1/q}}{ k_{i+1}-k_{i}} \Big[F_{k_i}^{1/\mu}+F_{k_i}^{ 1/(\mu-a)}\Big].
\eeqs
Hence
\beq\label{sigki}
\sigma_{k_{i+1}}\le C (1+T)\frac{ 2^{qi} }{M_0^q}\Big[ (\mathcal E_T \sigma_{k_i})^{ q/\mu}+ (\mathcal E_T \sigma_{k_i})^{q/(\mu-a)}\Big].
\eeq
Let 
\beqs 
Y_i=\sigma_{k_i},
\quad A_1=C(1+T) \mathcal E_T^{q/\mu} M_0^{-q},
\quad  A_2=C(1+T) \mathcal E_T^{q/(\mu -a)} M_0^{-q},
\quad B=2^q,
\eeqs
\beqs
\mu_0= q/ \mu-1,\quad \nu_0= q/(\mu-a)-1.
\eeqs
Then \eqref{sigki} rewrites as 
\beqs
Y_{i+1}\le B(A_1 Y_i^{1+\mu_0}+A_2 Y_i^{1+\nu_0}).
\eeqs
Note that $\mu_0<\nu_0$,
$
k_0=M_0\ge \|\bar p_0\|_{L^\infty},
\quad Y_0=\sigma_{M_0}\le |Q_T|= CT.
$

Choose $M_0$ large such that 
\beq\label{hyp}
 T+1\le C \min\Big\{ A_1^{-1/\mu_0},A_2^{-1/\nu_0} \Big\}.
\eeq
Explicitly,  
\beqs
\begin{split}
M_0&\ge C(1+T)^\frac{\mu_0+1}{q} \mathcal E_T^\frac1\mu=C(1+T)^\frac1\mu \mathcal E_T^\frac1\mu,\\
 M_0&\ge (1+T)^\frac{\nu_0+1}{q} \mathcal E_T^\frac 1{\mu-a}=(1+T)^\frac 1{\mu-a} \mathcal E_T^\frac 1{\mu-a}.
\end{split}
\eeqs
Since
\begin{align*}
 (1+T)\mathcal E_T &\le C (1+T)^{\mu} \|f-\Psi_t\|_{L_t^\infty (0,T;L_x^{\mu+1})}^\mu\\
 &\quad+C (1+T)^\frac{\mu}{2}(1+\|\nabla \Psi\|_{L_t^\infty(0,T; L_x^\infty)}^{\beta})^\frac{\mu}{2}.
\end{align*}
We select
\beqs
\begin{aligned}
M_0 =\|\bar p_0\|_{L^\infty}+ C\Big\{ (1+T)^{\mu} (1+\|f-\Psi_t\|_{L_t^\infty (0,T;L_x^{\mu+1})})^\mu \\
+(1+T)^\frac{\mu}{2}(1+\|\nabla \Psi\|_{L_t^\infty(0,T; L_x^\infty)}^{\beta})^\frac{\mu}{2}\Big\}^\frac1{\mu-a}.
\end{aligned}
\eeqs
Then \eqref{hyp} holds. Applying \ref{multiseq} with $m=2$, we have 
\beqs \sigma_{2M_0}=\lim_{i\to\infty} \sigma_{k_i}=0,\eeqs
that is
\beqs
\bar p(x,t)\le 2M_0\quad \text{a.e. in}\quad Q_T.
\eeqs
Replacing $p$ by $-p$, $\psi$ by $-\psi$. We finish the proof. 
\end{proof}

{\bf Acknowledgments.}
The authors is deeply grateful to Luan Hoang for precious advice, stimulating discussions, and helpful comments.

\bibliographystyle{siam}

\def\cprime{$'$} \def\cprime{$'$}

\end{document}